\documentclass[a4paper,12pt]{article}
\usepackage{amsmath,amsthm,amsfonts,amssymb,bbm}
\usepackage{graphicx,psfrag,subfigure,color,cite}

\numberwithin{equation}{section}

\usepackage[margin=1.2in]{geometry}

\newcommand{\id}{{\rm id}}
\newcommand{\e}{\varepsilon}
\newcommand{\Pb}{\mathbb{P}}
\newcommand{\E}{\mathbb{E}}

\newcommand{\R}{\mathbb{R}}

\newcommand{\Z}{\mathbb{Z}}

\newcommand{\Id}{\mathbbm{1}}

\newcommand{\Or}{{\cal O}}

\newtheorem{prop}{Proposition}[section]
\newtheorem{thm}[prop]{Theorem}
\newtheorem{lem}[prop]{Lemma}
\newtheorem{defin}[prop]{Definition}
\newtheorem{cor}[prop]{Corollary}

\newtheorem{cla}[prop]{Claim}

\newtheorem{rem}[prop]{Remark}
\newenvironment{remark}{\begin{rem}\normalfont}{\end{rem}}

\title{Shock fluctuations in TASEP\\ under a variety of time scalings}
\author{Alexey Bufetov\thanks{Institute for Applied Mathematics, Bonn University, Endenicher Allee 60, 53115 Bonn, Germany. E-mail: {\tt alexey.bufetov@gmail.com}}
\and Patrik L.\ Ferrari\thanks{Institute for Applied Mathematics, Bonn University, Endenicher Allee 60, 53115 Bonn, Germany. E-mail: {\tt ferrari@uni-bonn.de}}
}

\date{}

\begin{document}
\maketitle

\sloppy
\begin{abstract}
We consider the totally asymmetric simple exclusion process (TASEP) with two different initial conditions with shock discontinuities formed by blocks of fully packed particles. Initially a second class particle is at the left of a shock discontinuity. Using multicolored TASEP we derive exact formulas for the distribution of the second class particle and colored height functions. These are given in terms of the height function at different positions of a single TASEP configuration. We study the limiting distributions of second class particles (and colored height functions). The result depends on how the width blocks of particles scale with the observation time; we study a variety of such scalings.
\end{abstract}

\section{Introduction and main results}
The totally asymmetric simple exclusion process (TASEP) is a well-studied interacting particle system in the Kardar-Parisi-Zhang (KPZ) universality class. Each site of $\Z$ can be either occupied by a particle or be empty. Thus a configuration is an element $\eta\in \{0,1 \}^\Z$, where $\eta(x)=1$ if $x$ is occupied and $\eta(x)=0$ if site $x$ is empty. Independently of each other, particles try to jump to their right neighboring site with rate $1$. The jump occurs if the arriving site is empty. In this paper we investigate a so-called second-class particle, which is a particle which also tries to jump to its right with rate $1$, but whenever a normal particle (also called a first class particle) jumps on its position, then the second class particle exchanges its position with the one of the first class particle.

Second class particles can be also seen as discrepancies between two TASEP systems coupled by the basic coupling, see for instance~\cite{Li99} for further details. From this point of view, the distribution of the second class particle gives an information on the correlation between particle occupations. For instance, in the case of stationary initial condition the probability that the second class particle at time $t$ is at position $j$ is precisely proportional to ${\rm Cov}(\eta_t(j),\eta_0(0))$, see e.g.~\cite{PS01}. Finally, when the interacting particle system generates shocks, this can be identified with the position of the second class particle, see Chapter~3 of~\cite{Li99}.

Previously results on asymptotics of the position of a second class particle for some non-random initial conditions have been obtained by passing through the connection to a last passage percolation (LPP) model~\cite{FGN17}. However, the LPP framework is not the natural one to study the second class particle. In that framework the natural observable is the so-called competition interface. The trajectory of one second class particle is a (non-trivial) random time change of the competition interface ~\cite{FP05B},~\cite{FMP09}. The distribution of the competition interface can be expressed in terms of the differences of last passage times. Asymptotic results can be found in~\cite{FN13,FN16,N17}. Due to the random time change, it is not immediate to extend the result on the competition interfaces to the position of the second class particle. Another important point is that we also study TASEP with many second and even third class particles. It is not known whether these multi-species TASEP's can be coupled with LPP at all. In this paper we never use the mapping to LPP to derive our results.

To study the distribution of the position of the second class particle, in particular its large time limit, we employ the following generalization of TASEP~\cite{AAV11}. Each site has a particle with integer-valued colors (where $+\infty$ corresponds also to holes). Each particle tries to jump to its right with rate $1$ and the jump occurs only if the right site is occupied by a particle with higher color, in which case the two particles exchange their positions. This process has a rich algebraic structure (due to its connection to Hecke algebras~\cite{B20}) which we use for the study of the distribution of the position of the second class particles.

In this paper we consider two (non-random) initial conditions for TASEP in which one or two shocks are present, and we set a single second class particle starting at the boundary of a shock, as illustrated in Figure~\ref{FigIC}. The first new result of this paper is that the distribution of the position of the second class particle can be written in terms of a difference of the height function of TASEP (without second class particles) at two different positions, see Propositions~\ref{prop:shock1} and~\ref{prop:shock2}. This non-asymptotic result is achieved by relating to a system of colored particles with $3$ different colors in the first and $4$ different colors for the second initial condition. In Propositions~\ref{prop:shock1} and~\ref{prop:shock2} we also state a more general exact relation of the height functions for the different colors to a single TASEP in the situation when all particles inside one block are assigned its own color, so many second- and third-class particles are present in the system.
\begin{figure}
\begin{center}
\psfrag{0}[c]{$0$}
\psfrag{Mm}[c]{$-M_-$}
\psfrag{Mp}[c]{$M_+$}
\psfrag{M}[c]{$M$}
\psfrag{-M}[c]{$-M$}
\psfrag{MN}[c]{$M+N$}
\psfrag{-MN}[c]{$-M-N$}
\includegraphics[height=4cm]{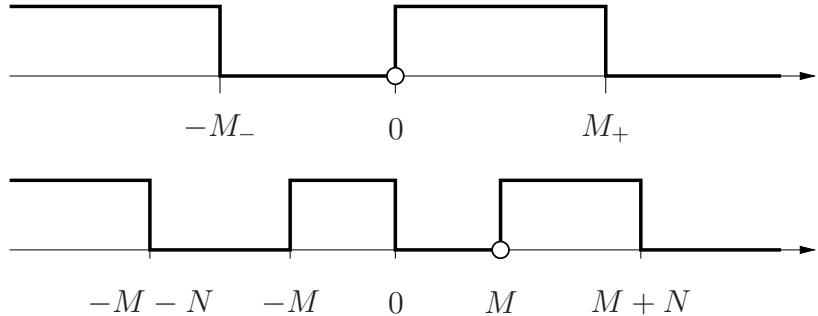}
\caption{Density of particles at time $t=0$ for the two initial conditions considered. The plateaux have density $1$ and the white dot is the position from where the second class particle starts.}
\label{FigIC}
\end{center}
\end{figure}

Let us explain the main asymptotic results for the first of the two initial conditions we considered. At time $t=0$, there are two blocks of fully occupied regions: to the left of $-M_-$ and between $1$ and $M_+$. A second class particle is put initially at position $0$, see Figure~\ref{FigOneGUEGUE}. The limiting distribution and scaling of the position of the second class particle depends on how $M_\pm$ scales with the observation time $t$. Here we consider three scalings of $M_\pm$ with time.
\begin{itemize}
\item[(a)] Let $M_+=[at]$ and $M_-=[bt]$, with $0<a,b<1$. Then the second class particle has fluctuations in the $t^{1/3}$ scale and its statistics is given as a difference of two independent GUE Tracy-Widom distributed random variables, see Proposition~\ref{prop5.1}.
\item[(b)] Let $M_+=[at^\delta]$ and $M_-=[bt^\delta]$, with $a,b>0$ and $\delta\in (2/3,1)$. Then the fluctuation scale is  $t^{4/3-\delta}$ and the statistics is a difference of two independent GUE Tracy-Widom distributed random variables, see Proposition~\ref{prop5.2}.
\item[(c)] Let $M_+=[at^\delta]$ and $M_-=[bt^\delta]$, with $a,b>0$ and $\delta\in (0,2/3)$. Then the fluctuation of the second class particles are Gaussian in a scale $t^{1-\delta/2}$, see Proposition~\ref{prop5.3}.
\end{itemize}
The $\delta=2/3$ case has been analyzed previously in~\cite{BB19}. See Section~\ref{sec:theorems1shock} for the detailed results and Section~\ref{sec:theorems2shock} for the result with the second initial condition.

The use of the algebraic structure behind the multi-colored process allows to relate the distribution of the position of the second class particle to a certain observable of a single-species TASEP started from the step initial condition. However, finding the asymptotic behavior of this observable is a non-trivial task which is addressed in Sections~\ref{sec:Geodesics} and \ref{sec:StepAsymp}.
Cases (a) and (b) above are similar as the limiting distribution is given in terms of two independent random variables. They however differ as for $\delta\in(2/3,1)$ there is one extra term in the law of large number approximation to be taken into account. The result is obtained using the following strategy. First of all we show that the randomness generated in the any initial mesoscopic time scale is irrelevant, see Proposition~\ref{PropSlowDec}. This property is essentially the same as the so-called slow-decorrelation one studied in~\cite{Fer08,CFP10b}. Secondly, we show that the randomness away of a $u t^{2/3}$ distance from the characteristic ending at $(x,t)$ can influence the height function at $(x,t)$ only with probability bounded from above by $e^{-c u^2}$, see Proposition~\ref{PropLocalization}, i.e., the relevant randomness is localized. This result is obtained by introducing backwards geodesics, which then implies the concatenation property for the height function (see Proposition~\ref{PropConcatenation}). After that we determine an estimate of the tail distribution of the geodesics at time $t/2$. Finally we use the approach developed in the LPP framework in~\cite{BSS14} for an uniform estimate over the full time span. All together this leads to the generic asymptotic decoupling result of Theorem~\ref{thm:DecorrelationHeights}.

This construction using the space-time picture for the height function is also novel. The space-time physical picture is expected to be true also in cases like the partially asymmetric simple exclusion process, although in that framework the concatenation property is only an inequality. A recent result on particle statistics at a shock with initial condition as the one of Figure~\ref{FigOneGUEGUE}, see~\cite{N19}. To be able to apply the ideas of this paper to the partially asymmetric case, one should first get a control on the deviations from the formula given by the concatenation property, which is not available so-far.

For case (c), the end-points where the TASEP height function is evaluated are smaller than the correlation scale, which is $t^{2/3}$. Thus we need to show that local increments are similar to the stationary ones. This is made using the comparison lemma (see Lemma~\ref{lemComparison}), see~\cite{CP15b,Pim17} for an analogous approach for LPP models. As input we need to know the locations at time $0$ of the backwards characteristics for different initial conditions. The result on the Gaussian increments is stated in Theorem~\ref{thm:LocalGaussianHeight}.

\paragraph{Outline:} In Section~\ref{sect:ExactRelations} we derive the finite time exact relations between position of the second class particle and colored height functions with the height function of TASEP with step initial conditions. The main asymptotic results are presented in Section~\ref{sectMainThmProof}. In Section~\ref{sec:Geodesics} we give the construction of the backwards geodesics, small time decorrelations and derive the localization results. Using these, we derive in Section~\ref{sec:StepAsymp} the asymptotic results for the height function of TASEP with step initial conditions needed for our main asymptotic theorems.

\paragraph{Acknowledgments: } We are grateful to A.~Borodin for very useful duscussions on the early stages of the project. The work of P.L. Ferrari is supported by the German Research Foundation through the Collaborative Research Center 1060 ``The Mathematics of Emergent Effects'' (Projekt ID 211504053), project B04. Furthermore, the work of both authors is supported by the Deutsche Forschungsgemeinschaft (DFG, German Research Foundation) under Germany's Excellence Strategy - GZ 2047/1, Projekt ID 390685813.

\section{Exact relations to TASEP with step initial conditions}\label{sect:ExactRelations}

\subsection{Updates of multi-colored TASEP and symmetry theorem}
We start with a description of a colored (or multi-species, or multi-type) version of the totally asymmetric simple exclusion process (TASEP). We consider an interacting particle system in which particles live on the integer lattice $\Z$ and each integer location contains exactly one particle. The set of colors is taken as $\Z \cup \{ +\infty \}$.

A particle configuration is a map $\eta: \Z \to \Z \cup \{ +\infty \}$, where we call $\eta(z)$ the \emph{color} of a particle at $z\in\Z$. When $\eta(z)=+\infty$ we will think of $z$ being empty. Let $\mathfrak C$ be the set of all configurations. For a transposition $(z,z+1)$ with $z, z+1 \in \Z$, let $\sigma_{(z,z+1)}: \mathfrak C \to \mathfrak C$, be a swap operator defined by
\begin{equation}
(\sigma_{(z,z+1)} \eta) (i) =
\begin{cases}
\eta(z+1), \qquad & i=z, \\
\eta(z), \qquad & i=z+1, \\
\eta(i), \qquad & i \in \Z \backslash \{ z,z+1 \}.
\end{cases}
\end{equation}
Define a \textit{totally asymmetric swap} operator $W_{(z,z+1)}: \mathfrak C \to \mathfrak C$, $z \in \Z$, via
\begin{equation}
(W_{(z,z+1)} \eta) =
\begin{cases}
\eta, \qquad & \mbox{if $\eta(z) \ge \eta(z+1)$}, \\
\sigma_{(z,z+1)} \eta, \qquad & \mbox{if $\eta(z) < \eta(z+1)$}.
\end{cases}
\end{equation}

Any bijection of integers $s: \Z \to \Z$ can be viewed as a particle configuration by setting $\eta(z)=s(z)$. Such particle configurations will be especially important for us because of the following result.

\begin{thm}[Lemma 2.1 of \cite{AHR09}]
\label{th:sym-basic}
Let $\mathrm{id} : \Z \to \Z$ be the identity bijection. Then, for any $k \in \Z$ and for any integers $z_1, \dots, z_k$ one has
\begin{equation}
W_{(z_k,z_k+1)} \dots W_{(z_2,z_2+1)} W_{(z_1,z_1+1)} \id= \mathrm{inv} \left( W_{(z_1,z_1+1)} W_{(z_2,z_2+1)} \dots W_{(z_k,z_k+1)} \id \right),
\end{equation}
where in the right-hand side $\mathrm{inv}$ denotes the inverse map in the space of bijections $\Z \to \Z$.
\end{thm}
\begin{remark}
\label{rem:perm-as-particles}
Consider a bijection $\eta: \Z \to \Z$ interpreted as a particle configuration in the way described above: As a map from positions to colors of particles standing at these positions. Then ${\rm inv}(\eta)$ maps colors of particles to positions where they stand in the configuration $\eta$. For example, $\eta(0)$ is the color of the particle standing at 0, while ${\rm inv}(\eta)(0)$ is the position of the particle of color $0$ in the configuration $\eta$. Here and below we will denote by ${\rm inv}(\eta)(z)$ the application of the map ${\rm inv}(\eta)$ to an integer $z$.

\end{remark}

In a probabilistic setting, Theorem \ref{th:sym-basic} was proved in~\cite[Lemma 2.1]{AHR09}, see~\cite{AAV11} and~\cite{BB19} for generalizations. In an equivalent algebraic setting, it turns out to be a well-known involution in the Hecke algebra, see \cite{B20}, \cite{G20}.

\subsection{A continuous time multi-color TASEP}
\label{sec:defTasepMC}

Now let us define a continuous-time TASEP. Consider a collection of independent Poisson processes $\{ \mathcal P (z) \}_{z \in \Z}$, where $\mathcal P (z)$ has a state space $\R_{\ge 0}$ and rate $1$. Let $\eta_0 \in \mathfrak C$ be a (either deterministic or random) particle configuration which plays the role of an initial condition. In the random case, the initial distribution is taken to be independent of the Poisson processes. We define a continuous-time stochastic evolution $\{ \eta_t \}_{t \in \R_{\ge 0}}$, $\eta_t \in \mathfrak C$, by applying $W_{(z,z+1)}$ at time $t$ which is an event of the Poisson process $\mathcal P (z)$. More explicitly, the particle at $z$ exchanges its position with the particle at site $z+1$ if its color has a lower value. It is readily shown via standard techniques that under our assumptions such a random process is well-defined, see~\cite{Har78,Har72,Hol70,Lig72}. Denote by $S(t)$ the semigroup of the process.

Given a sequence of nearest-neighbor transpositions $s_1, s_2, \dots, s_l$, we will consider two processes associated with this data. In the first one, let us start with the initial configuration $\id(z)=z$, then apply to it updates $W_{s_1}, W_{s_2},\ldots, W_{s_l}$ ($W_{s_1}$ is the first update to be applied). After it, we start the continuous time process described in the previous paragraph and denote by $\eta_{t}^{{\rm gen};1}$ the random configuration obtained after time $t$. For the second process, we start again with the initial configuration $\id(z)=z$, but first perform the continuous time process running for time $t$. After this, we apply to the resulting (random) configuration the updates in the reversed order, that is, $W_{s_l}, W_{s_{l-1}}, \ldots, W_{s_1}$. Denote by $\eta_{t}^{{\rm gen};2}$ the obtained configuration. The superscript ``gen'' stands for generic since each particle has a different color. The processes considered in the sequel will be projections of these processes, where sets of colors will be mapped to a single particle type.

\begin{prop}[Theorem 3.1 of~\cite{BB19}]
\label{prop:SymConts}
The random configurations $\eta_{t}^{{\rm gen};1}$ and $\mathrm{inv} \left( \eta_{t}^{{\rm gen};2} \right)$ have the same distribution, i.e.,
\begin{equation}
S(t)W_{s_l}\cdots W_{s_1} \id \stackrel{d}{=}{\rm inv}(W_{s_1}\cdots W_{s_l} S(t) \id ),
\end{equation}
and, of course, also $\mathrm{inv}(\eta_{t}^{{\rm gen};1} )\stackrel{d}{=}\eta_{t}^{{\rm gen};2} $.
\end{prop}

\begin{remark}
This proposition is a direct corollary of Theorem \ref{th:sym-basic} and the time-reflection symmetry of a homogeneous Poisson process. See the proof of \cite[Theorem 3.1]{BB19} for details (and also the proof of \cite[Theorem 1.4]{AAV11} for a similar statement).
\end{remark}


\subsection{Standard TASEP with step initial conditions}

Consider a continuous time two-color TASEP $\eta_t^{\rm step} (z)$ starting from the step initial condition
\begin{equation}\label{eq2.5}
\eta_0^{\rm step} (z) = \begin{cases}
1, \quad & z \le 0, \\
+\infty, \quad & z >0.
\end{cases}
\end{equation}
For any $x \in \R$, we denote by $\mathcal{N} (x,t)$ the number of particles that are weakly to the right of $x$ in $\eta_t^{\rm step} (z)$,
\begin{equation}
{\cal N}(x,t)=\sum_{z\geq x}\delta_{\eta_t^{\rm step}(z),1}.
\end{equation}

We will relate other processes to this simpler and better studied one.

\subsection{One GUE-GUE shock}

In this section we will study a homogeneous three-color TASEP with a particular initial condition, see Figure~\ref{FigOneGUEGUE}.
\begin{figure}
\begin{center}
\psfrag{0}[c]{$0$}
\psfrag{Mm}[c]{$-M_-$}
\psfrag{Mp}[c]{$M_+$}
\psfrag{eta1}[l]{$\eta_0^{(1)}$}
\psfrag{eta2}[l]{$\tilde\eta_0^{(1)}$}
\psfrag{1st}[lb]{= first class}
\psfrag{2nd}[lb]{= second class}
\includegraphics[height=3cm]{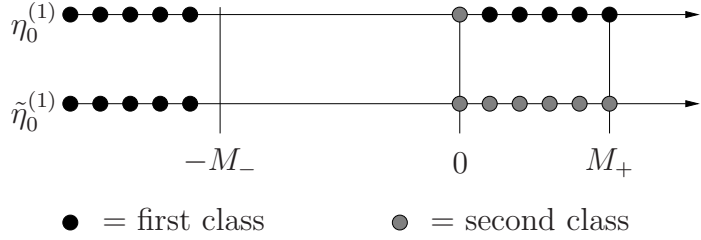}
\caption{Initial particle configurations $\eta_0^{(1)}$ and $\tilde \eta_0^{(1)}$. The particles with values $+\infty$, also considered as holes, are not shown.}
\label{FigOneGUEGUE}
\end{center}
\end{figure}

Let $M_{-}, M_{+}$ be positive integers, and consider a TASEP denoted as $\eta_t^{(1)} (z)$ with the initial condition
\begin{equation}
\eta_0^{(1)} (z) = \begin{cases}
1, \quad & z <-M_{-}\textrm{ or }1 \le z \le M_{+}, \\
2, \quad & z=0, \\
+\infty, \quad &\textrm{otherwise}
\end{cases}
\end{equation}
Also, define a TASEP denoted as $\tilde \eta_t^{(1)} (z)$ with the initial condition
\begin{equation}
\tilde \eta_0^{(1)} (z) = \begin{cases}
1, \quad & z <-M_{-}, \\
2, \quad & 0 \le z \le M_{+}, \\
+\infty, \quad & \textrm{otherwise}.
\end{cases}
\end{equation}

Both these processes illustrate the so called GUE-GUE shock. We will study the first process in order to analyze the behavior of the unique second class particle in the shock, and the second process in order to analyze the behavior of the multi-colored height function in the shock.

In more detail, let $\mathfrak{f}^{(1)} (t)$ be the position of the unique second class particle in the process $\eta_t^{(1)} (z)$.
Let $\mathfrak{N}^{(1)}_{1} (x,t)$ be the number of particles of color $1$ in $\tilde \eta_t^{(1)}$ which are weakly to the right of $x \in \R$. Let $\mathfrak{N}^{(1)}_{2} (x,t)$ be the number of particles of color $2$ in $\tilde \eta_t^{(1)}$ which are weakly to the right of $x \in \R$, namely
\begin{equation}
\mathfrak{N}^{(1)}_{c} (x,t) =\sum_{z\geq x}\delta_{\tilde \eta_t^{(1)}(z),c}, \qquad \mbox{for $c=1,2$.}
\end{equation}

We will study these quantities by relating them to a simpler process via the color-position symmetry. The first claim (equation \eqref{eq:SchockExact1}) of the following proposition is a minor generalization of \cite[Proposition 6.1]{BB19}, where only the case $M_{-}=M_{+}$ was addressed.

\begin{prop}
\label{prop:shock1}
For any $x \in \Z$ we have
\begin{equation}
\label{eq:SchockExact1}
\Pb\left( \mathfrak{f}^{(1)} (t) \ge x \right) = \Pb\left( \mathcal{N} ( x - M_{+},t) -  \mathcal{N} (x + M_{-}+1 ,t) \ge M_{+} +1 \right).
\end{equation}
Also, for any $x \in \Z$ we have
\begin{equation}
\begin{aligned}
\label{eq:SchockExact11}
&\left(  \mathfrak{N}^{(1)}_{1} (x,t) , \mathfrak{N}^{(1)}_{2} (x,t)  \right) \\
\stackrel{d}{=}  &\left( \mathcal{N} (x+M_{-}+1, t) , \min \left\{ \mathcal{N} (x - M_{+}, t) - \mathcal{N} (x + M_{-}+1, t), M_{+}+1 \right\} \right),
\end{aligned}
\end{equation}
where by $\stackrel{d}{=}$ we denote the equality in distribution.

\end{prop}

\begin{proof}

Let $\pi_{-M_-, M_+}$ be a permutation of the set $\{-M_{-}, \dots, M_{+} \}$ such that $\pi_{-M_-, M_+} (-M_{-} + i) = M_{+} -i$, for $i=0,1,\dots, M_{+} + M_{-}$.
Consider a minimal length decomposition of this permutation into transpositions of neighbouring elements: $\pi_{-M_-, M_+} = s_{m} s_{m-1} \dots s_2 s_1$ (there are many such decompositions, we choose any of them; we always have $m=(M_{+} + M_{-})(M_{+} + M_{-}+1)/2$, since this is the number of inversions in $\pi_{-M_-, M_+}$ ).

The following remark will be important for us in the proof. Consider an arbitrary (infinite) permutation $\mathbf{\tilde \sigma}: \Z \to \Z$, and its interpretation as a multi-color configuration on $\Z$ (see Remark \ref{rem:perm-as-particles}). Then the permutation
$W_{s_m}\cdots W_{s_1} \mathbf{\tilde \sigma}$ is the following. Outside of the set $\{-M_{-}, \dots, M_{+} \}$ it coincides with $ \mathbf{\tilde \sigma}$. Inside the set $\{-M_{-}, \dots, M_{+} \}$ it is obtained by sorting all particles from $\mathbf{\tilde \sigma}$ standing there in the decreasing order of their colors. For example,
$$
W_{s_m} \cdots W_{s_1} \mathbf{\tilde \sigma} (-M_{-}) = \max_{-M_{-} \le i \le M_+} \mathbf{\tilde \sigma}(i), \qquad
W_{s_m} \cdots W_{s_1} \mathbf{\tilde \sigma} (M_{+}) = \min_{-M_{-} \le i \le M_+} \mathbf{\tilde \sigma}(i).
$$

Consider the process constructed from the identity by applying the transpositions $s_1, \dots, s_m$, i.e., $\eta_{0}^{{\rm gen};1}(z)=W_{s_m}\cdots W_{s_1} \id(z)$.
In words, in the packed initial configuration we sort integers between $-M_{-}$ and $M_+$ in the reverse order, i.e., the $\eta_0^{{\rm gen};1}$ configuration in the interval $[-M_-,M_+]$ is $(M_+,M_+-1,\ldots,-M_-)$. Then we run the continuous time dynamics until time $t$ (see the definition of the processes $\eta_{t}^{{\rm gen};1}$ and $\eta_{t}^{{\rm gen};2}$ in Section \ref{sec:defTasepMC}).

The interpretation in terms of first and second class particle is the following. Since $\pi_{-M_-, M_+} (-M_{-}+M_{+}) =0$, we identify the particle with color $-M_-+M_+$ as the second class particle. Furthermore, particles with colors $<-M_{-}+M_{+}$ are called first class particles, and particles with colors $>-M_{-}+M_{+}$ are called holes.


This gives
\begin{equation}
\Pb\left( \mathfrak{f}^{(1)} (t) \ge x \right) = \Pb \left( \mathrm{inv} \left( \eta_{t}^{{\rm gen};1} \right) (-M_{-}+M_{+}) \ge x \right).
\end{equation}
Similarly, interpreting colors $<-M_{-}$ as first class particles, colors between $-M_{-}$ and $-M_{-}+M_{+}$ as second class particles, and colors $>-M_{-}+M_{+}$ as holes, we obtain
\begin{multline}
\left(  \mathfrak{N}^{(1)}_{1} (x,t) , \mathfrak{N}^{(1)}_{2} (x,t)  \right) \stackrel{d}{=} \left( \#\textrm{ of }\left[ i: i< -M_{-}, \mathrm{inv} \left( \eta_{t}^{{\rm gen};1} \right) (i) \ge x \right], \right. \\ \left. \#\textrm{ of } \left[ i: -M_{-} \le i \le -M_{-}+M_{+}, \mathrm{inv} \left( \eta_{t}^{{\rm gen};1} \right) (i) \ge x \right] \right).
\end{multline}

Proposition~\ref{prop:SymConts} gives
\begin{equation}
\Pb \left( \mathrm{inv} \left( \eta_{t}^{{\rm gen};1} \right) (-M_-+M_+) \ge x \right) = \Pb \left( \eta_{t}^{{\rm gen};2} (-M_{-}+M_{+}) \ge x \right),
\end{equation}
and
\begin{equation}
\begin{aligned}
&\left( \#\textrm{ of }\left[ i: i< -M_{-}, \mathrm{inv} \left( \eta_{t}^{{\rm gen};1} \right) (i) \ge x \right], \right. \\
&\quad\quad \left. \#\textrm{ of }\left[ i: -M_{-} \le i \le -M_{-}+M_{+}, \mathrm{inv} \left( \eta_{t}^{{\rm gen};1}  \right)(i) \ge x \right] \right) \\
\stackrel{d}{=} &\left( \#\textrm{ of }\left[ i: i< -M_{-}, \eta_{t}^{{\rm gen};2} (i) \ge x \right],\right.\\
&\quad\quad \left.\#\textrm{ of }\left[ i: -M_{-} \le i \le -M_{-}+M_{+}, \eta_{t}^{{\rm gen};2} (i) \ge x \right] \right).
\end{aligned}
\end{equation}

Now we consider a second projection from $\eta_t^{{\rm gen};2}$ to particles and holes only. We say that if a color is $<x$, then we have a particle, while if the color is $\geq x$, then the site is empty. In particular, at time $t=0$ we have the step initial condition centered at $x-1$. After running the dynamics we have a configuration $S(t)\id(z)$. The application of $W_{s_1} \dots W_{s_m}$ reorders the colors in the interval $\{-M_-,\ldots,M_+\}$ in a decreasing order (since $s_1 \dots s_m$ also equals $\pi_{-M_-, M_+}$). Therefore, after this sorting all colors $\geq x$ are to the left of the colors $<x$. Thus, the color $\eta_{t}^{{\rm gen};2} (-M_{-}+M_{+})$ is greater or equal to $x$ whenever the number of holes in the set $\{-M_-,\ldots,M_+\}$ is at least $M_{+}+1$. Using the particle-hole duality of TASEP, this number of holes can be written as a difference of ${\cal N}$, namely
\begin{equation}
\Pb \left( \eta_{t}^{{\rm gen};2} (-M_{-}+M_{+}) \ge x \right) = \Pb \left( \mathcal{N} ( x - M_{+},t) -  \mathcal{N} (x + M_{-}+1 ,t) \ge M_{+} +1 \right).
\end{equation}
Similarly we get
\begin{multline}
\left( \mbox{\# of $\left[ i: i< -M_{-}, \eta_{t}^{{\rm gen};2} (i) \ge x \right]$}, \mbox{\# of $\left[ i: -M_{-} \le i \le -M_{-}+M_{+}, \eta_{t}^{{\rm gen};2} (i) \ge x \right]$} \right) \\ \stackrel{d}{=} \left( \mathcal{N} (x+M_{-}+1, t) , \min \left\{ \mathcal{N} (x - M_{+}, t) - \mathcal{N} (x + M_{-}+1, t), M_{+}+1 \right\} \right).
\end{multline}
This concludes the proof of the proposition.
\end{proof}

\subsection{Two GUE-GUE shocks}
Now we consider initial conditions as in Figure~\ref{FigTwoGUEGUE}.
\begin{figure}
\begin{center}
\psfrag{M1}[c]{$-M-N$}
\psfrag{M2}[c]{$-M$}
\psfrag{M3}[c]{$0$}
\psfrag{M4}[c]{$M$}
\psfrag{M5}[c]{$M+N$}
\psfrag{eta1}[l]{$\eta_0^{(2)}$}
\psfrag{eta2}[l]{$\tilde\eta_0^{(2)}$}
\psfrag{1st}[lb]{= first class}
\psfrag{2nd}[lb]{= second class}
\psfrag{3rd}[lb]{= third class}
\includegraphics[height=3cm]{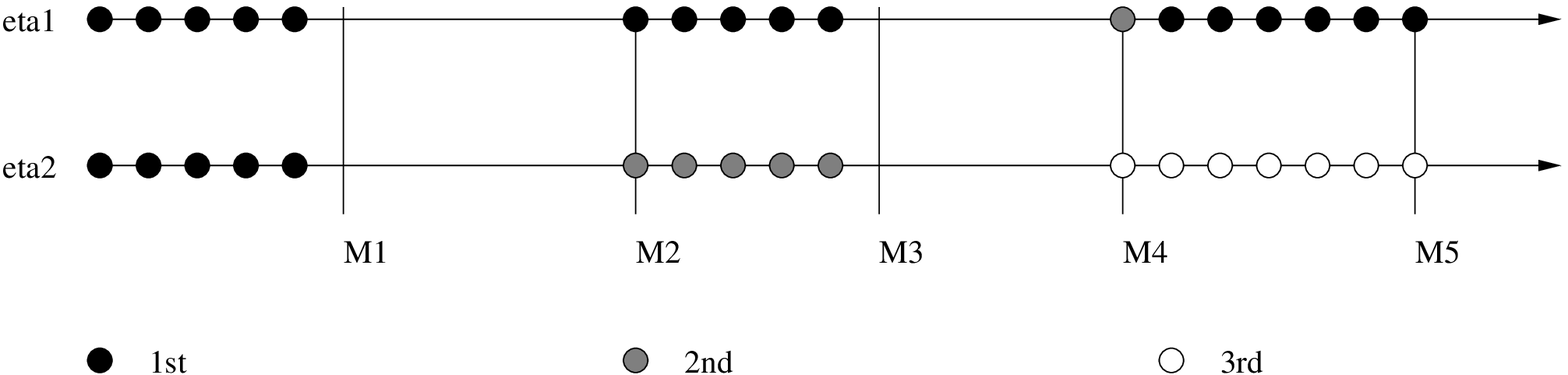}
\caption{Initial particle configurations $\eta_0^{(2)}$ and $\tilde \eta_0^{(2)}$. The particles with values $+\infty$, also considered as holes, are not shown.}
\label{FigTwoGUEGUE}
\end{center}
\end{figure}
Let $M, N$ be positive integers, and consider a TASEP denoted as $\eta_t^{(2)} (z)$ with the initial condition
\begin{equation}
\eta_0^{(2)} (z) = \begin{cases}
1, \quad & z <-M - N\textrm{ or } - M \le z \le -1\textrm{ or }M+1 \le z \le M + N,\\
2, \quad & z = M, \\
+\infty, \quad & \textrm{otherwise}.
\end{cases}
\end{equation}
Also, define a TASEP denoted as $\tilde \eta_t^{(2)} (z)$ with the initial condition
\begin{equation}
\tilde \eta_0^{(2)} (z) = \begin{cases}
1, \quad & z <-M - N, \\
2, \quad & - M \le z \le -1, \\
3, \quad & M \le z \le M + N, \\
+\infty, \quad & \textrm{otherwise}.
\end{cases}
\end{equation}

Let $\mathfrak{f}^{(2)} (t)$ be the position of the unique second class particle in the process $\eta_t^{(2)} (z)$.
Let $\mathfrak{N}^{(2)}_{i} (x,t)$ be the number of particles of color $i$ in $\tilde \eta_t^{(2)}$ which are weakly to the right of $x \in \R$, where $i \in \{1,2,3\}$.

\begin{prop}
\label{prop:shock2}
For any $x \in \Z$ we have
\begin{multline}
\label{eq:SchockExact2}
\Pb \left( \mathfrak{f}^{(2)} (t) \ge x \right) = \Pb( \mathcal{N} (x-M-N,t) - \mathcal{N} (x+1, t)\\ + \max \left\{ 0, \mathcal{N}(x+1, t) - \mathcal{N}(x+M+N+1, t) -M \right\} \ge N+1).
\end{multline}
Also, for any $x \in \Z$ we have
\begin{equation}
\begin{aligned}
\label{eq:SchockExact22}
&\left(  \mathfrak{N}^{(2)}_{1} (x,t) , \mathfrak{N}^{(2)}_{2} (x,t), \mathfrak{N}^{(2)}_{3} (x,t)  \right) \\ \stackrel{d}{=}&  \left( \mathcal{N} (x+M +N +1, t) , \min \left\{ \mathcal{N} (x +1, t) - \mathcal{N} (x + M+N+1, t), M \right\}, \right.  \\ &\quad\min \big\{\mathcal{N} (x-M-N,t) - \mathcal{N} (x+1, t) \\ &\left.  \quad\quad+ \max \left\{ 0, \mathcal{N}(x+1, t) - \mathcal{N}(x+M+N+1, t) -M \right\}, N+1 \big\} \right).
\end{aligned}
\end{equation}

\end{prop}

\begin{proof}

In the proof of Proposition~\ref{prop:shock1} we did one ``sorting'' operation (in the interval between $-M_{-}$ and $M_{+}$), now we will do two of them. First, we sort all colors on the interval between $-N$ and $M+N$ in the opposite order. Second, we sort all colors on the interval between $-M-N$ and $-1$ in the opposite order.

If we say that negative colors are first class particles, color 0 is the second class particle, and positive colors are holes, we obtain the initial condition $\eta_0^{(2)} (z)$ as the result of this procedure.

If we say that colors $<-M-N-1$ are first class particles, colors between $-M-N$ and $-M-1$ are second class particles, colors between $-M$ and $0$ are third class particles, and positive colors are holes, we obtain the initial condition $\tilde \eta_0^{(2)} (z)$.

The analysis of the reversed time process (which makes these two sorting operations at the end of the continuous time process) is analogous to the proof of Proposition~\ref{prop:shock1}.

\end{proof}

\section{Main asymptotic results}\label{sectMainThmProof}

\subsection{Macroscopic picture}
Under hydrodynamic scaling the evolution of the particle density is a solution of the Burgers equation. Let us illustrate the macroscopic density as well the evolution of the shocks for the setting as in Proposition~\ref{prop5.1} and~\ref{prop5.2}.

In the setting of Proposition~\ref{prop5.1}, initially a bloc of particle occupy $(-\infty,-b t]$ and a second bloc starts from $[0,at]$. Our statement is about the fluctuation of the second class particle at time $t$.

\emph{Case 1: $0<a<b<1$.} The particle starting at the origin has already moved before the particle which started at $-bt$ arrives. Let $\tau$ denote the macroscopic time variable and $\xi$ the macroscopic position, that is, the microscopic time it $\tau t$ and the microscopic position is $\xi t$, where we think of $t\gg 1$.
If $-b+\tau=(\sqrt{\tau}-\sqrt{a})^2$, then particles starting from the left of the origin just reaches the block of particles started from the right of the origin. Therefore the shock starts developing at $\tau =\frac{(a+b)^2}{4a}$ at position $\xi=\frac{(a-b)^2}{4a}$. For any later time, the position of the shock is $\frac{(a-b)(a+b-2\tau)}{2(a+b)}t$ and the density has a discontinuity of height $\frac{a+b}{2\tau}$. By choosing $\tau=1$ we get the condition $b<2\sqrt{a}-a$.

\emph{Case 2: $0<b<a<1$.} In this case, at time $\tau=b$ particles from the left bloc of particles reaches the particle starting from the origin, which did not move yet. Effectively this corresponds of enlarging the first block of particles. At time $\tau=\frac{(a+b)^2}{4b}$ the density on the right side of the shock starts decreasing from $1$. Since we look at time $\tau=1$ we get the condition $b>2-a-2\sqrt{1-a}$.

In Figure~\ref{FigCharacteristicsOneShock} for an illustration of the particle densities at time $t$ for $0<a<b<1$.  The picture for $0<b\leq a<1$ with $b>2-a-2\sqrt{1-a}$ is similar, except that the position of the shock is to the left of the origin.
\begin{figure}
\begin{center}
\psfrag{b}[c]{$-b t$}
\psfrag{a}[c]{$a t$}
\psfrag{0}[c]{$0$}
\psfrag{vt}[l]{$v t$}
\psfrag{ct}[l]{$c t$}
\psfrag{rho}[c]{$\rho$}
\includegraphics[height=3cm]{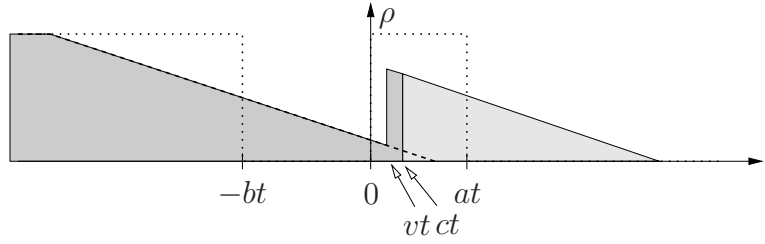}
\caption{Hydrodynamics for density of particle $\rho$ at time $t$ for one shock for the case $0<a< b<1$ with $b<2\sqrt{a}-a$. The dotted lines represent the initial condition. The particle which started at the origin is around $ct= (1-\sqrt{a})^2 t$, and the shock is around $vt=\frac{(a-b)(a+b-2)}{2(a+b)}t$. The two shaded regions correspond to the particles starting from the two blocks.}
\label{FigHydrodynamicsOneShock}
\end{center}
\end{figure}
In Figure~\ref{FigCharacteristicsOneShock} we illustrate the evolution of the shock for $0<a<b<1$.  For $0<b\leq a<1$, the picture is flipped at the imaginary axis, due to particle-hole symmetry.
\begin{figure}
\begin{center}
\psfrag{b}[c]{$-b t$}
\psfrag{a}[c]{$a t$}
\psfrag{0}[c]{$0$}
\psfrag{x}[c]{space}
\psfrag{t}[rb]{time}
\psfrag{shock}[l]{Shock}
\includegraphics[height=4cm]{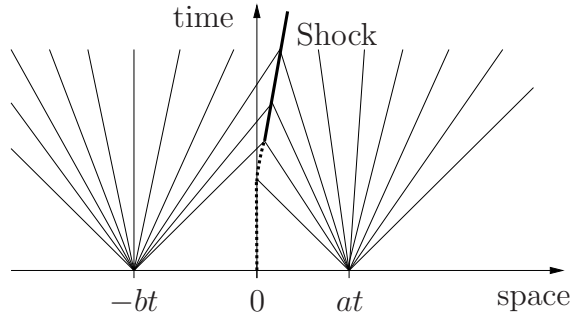}
\caption{Characteristics and shock for the one-shock case for $0<a<b<1$. The shock starts developing at macroscopic time $\frac{(a+b)^2}{4a}$ at macroscopic position $\frac{(a-b)^2}{4a}$. The position of the shock at macroscopic time $\tau$ is at macroscopic position $\frac{(a-b)(a+b-2\tau)}{2(a+b)}$ and the density has a discontinuity of height $\frac{a+b}{2\tau}$.}
\label{FigCharacteristicsOneShock}
\end{center}
\end{figure}

The hydrodynamic for the case of Proposition~\ref{prop5.2}, is similar until the time where the two shocks merges into a single shock. This happens if  $2(m-n)+(m+n)^2=0$, which corresponds to $n = 1 - m - \sqrt{1- 4m}$. Notice that if $m>1/4$, the two shocks did not have a chance to merge yet. In Figure~\ref{FigCharacteristicsTwoShocks} we illustrate the density of particles at time $t$. Note that the macroscopic picture of the density for $2(m-n)+(m+n)^2=0$ is the same as the one for $2(m-n)+(m+n)^2<0$, but the fluctuations are different.
\begin{figure}
\begin{center}
\psfrag{mmn}[c]{$-(m+n)t$}
\psfrag{mm}[c]{$-m t$}
\psfrag{0}[c]{$0$}
\psfrag{m}[c]{$m t$}
\psfrag{mn}[c]{$(m+n) t$}
\psfrag{rho}[c]{$\rho$}
\includegraphics[height=5cm]{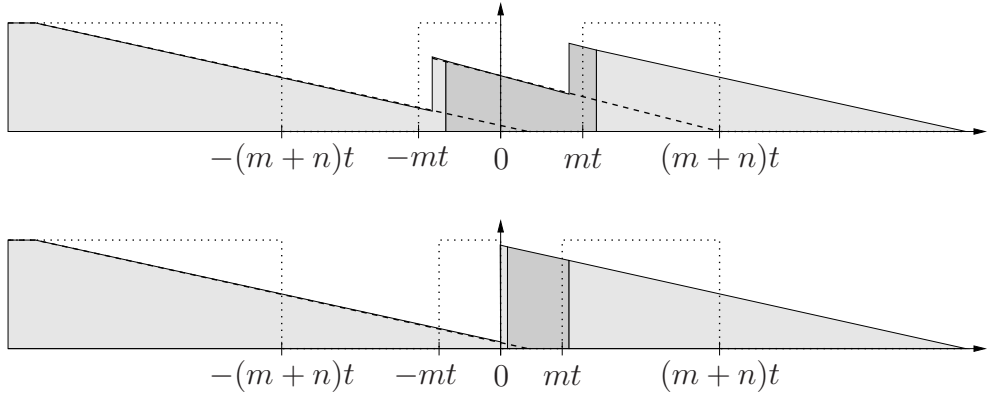}
\caption{Hydrodynamics for density of particle $\rho$ at time $t$ for two shocks. The dotted lines represent the initial condition. Here we have $0<m+n<1$ and $0<m<n<1$. The top picture is for $2(m-n)+(m+n)^2>0$. There are two shocks, the right one being the one we analyzed. The bottom picture is for $2(m-n)+(m+n)^2\leq 0$, where the two shocks have already merged.}
\label{FigHydrodynamicsTwoShocks}
\end{center}
\end{figure}

To see the difference, it is useful to look at the space-time picture of the characteristics and of the evolution of the shocks. Indeed, at the space-time point where the two shocks merge there are three incoming characteristic lines, while before and later at a shock position only two characteristics merge, see Figure\ref{FigCharacteristicsTwoShocks}.
\begin{figure}
\begin{center}
\psfrag{mmn}[c]{$-(m+n)t$}
\psfrag{mm}[c]{$-m t$}
\psfrag{0}[c]{$0$}
\psfrag{m}[c]{$m t$}
\psfrag{mn}[c]{$(m+n) t$}
\psfrag{x}[c]{space}
\psfrag{t}[rb]{time}
\includegraphics[height=6cm]{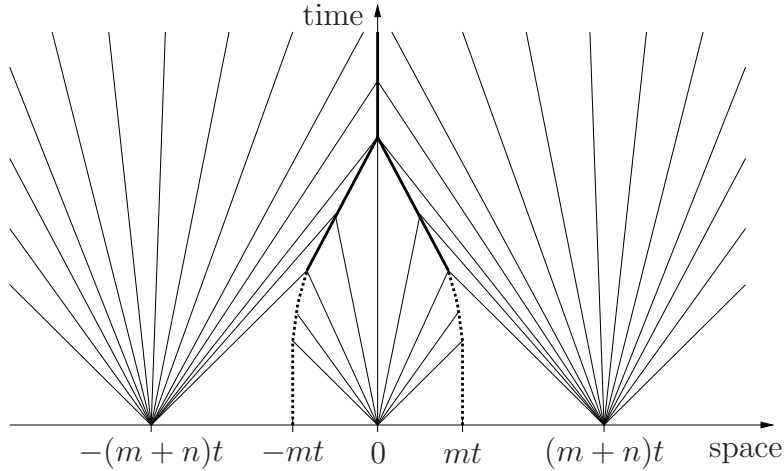}
\caption{Characteristics and shock for the one-shock case for $0<m+n<1$ and $0<m<n<1$. First two separate shocks are created and at macroscopic time $\frac{(m+n)^2}{2(n-m)}$ they meet at the origin. After that there is a single shock remaining at position $0$.}
\label{FigCharacteristicsTwoShocks}
\end{center}
\end{figure}

\subsection{One GUE-GUE shock}\label{sec:theorems1shock}

In this section we prove the results about the asymptotic behavior of one second class particle and the collection of second class particles in the case of one shock. There are four scalings in which the sizes of a block are proportional to $t^{\delta}$, with $\delta=1$, $\delta \in (2/3,1)$, $\delta = 2/3$, $\delta \in (0,2/3)$. Here we prove the asymptotic behavior for the first, second, and the fourth cases (the third case was analyzed previously in~\cite{BB19}).

Let us start with the case when the size of a block grows linearly in time.

\begin{prop}\label{prop5.1}
Let $a,b$ be two reals between 0 and 1, let $M_{+} = \lfloor a t \rfloor$, $M_{-} = \lfloor b t \rfloor$. Assume that $2-a-2\sqrt{a}<b<2\sqrt{a}-a$ and let $v :=\frac{(a-b) (a+b-2)}{2(a+b)}$. For any $s \in \R$,  we have
\begin{equation}
\lim_{t \to \infty}\frac{\mathfrak{f}^{(1)} (t) - vt}{t^{1/3}}\stackrel{d}{=}\frac{2}{a+b}(c_1 \xi_1 - c_2 \xi_2),
\end{equation}
and
\begin{equation}
\begin{aligned}
\label{eq:4444}
\lim_{t \to \infty} &\left(  \frac{\mathfrak{N}^{(1)}_{1} ( v t + s t^{1/3} ,t) - \frac{(1-v-b)^2}{4} t + \frac{(1-v-b)}{2} st^{1/3} }{t^{1/3}} , \frac{\mathfrak{N}^{(1)}_{2} ( v t + s t^{1/3} ,t) - a t}{t^{1/3}}  \right) \\ \stackrel{d}{=} &\left( - c_1 \xi_1 , \min \left( c_1 \xi_1 - c_2 \xi_2 - \tfrac{(a+b)s}{2} , 0 \right) \right),
\end{aligned}
\end{equation}
where $c_1=\frac{(1-(v+b)^2)^{2/3}}{2^{4/3}}$, $c_2=\frac{(1-(v-a)^2)^{2/3}}{2^{4/3}}$, and $\xi_1$ and $\xi_2$ are two independent GUE Tracy-Widom distributed random variables.
\end{prop}

\begin{remark}
\label{rem:4444}
In words, the result tells us the following. Let us start with one second class particle scenario. The second class particle indicates the position of the shock, which informally can be thought of as the position in which the particles from the left (infinite) block start to get affected by the presence of the right block. On the law of large numbers scale, this position has nontrivial speed $v$, which can be obtained from the hydrodynamics for the two blocks. We are interested in fluctuations of this position. These fluctuations are affected by two Tracy-Widom fluctuations of our two blocks. In the current scaling these fluctuations turn out to be asymptotically independent.

Our second scenario is to treat all particles from the right block as second class particles. We study the limit behavior of the joint distribution of the multi-colored counting functions in the neighborhood of the shock. The fluctuations are again governed by two independent Tracy-Widom distributions, and we see two different cases. In the first one, the last second class particle is to the right of the reference point --- this corresponds to the case when the minimum in the right-hand side of \eqref{eq:4444} is attained at 0. In the second case, when the minimum is attained not at 0, some second class particles are to the left of the reference point. Our result asserts that the amount of such particles is of order $t^{1/3}$ and, moreover, gives the precise distribution of this amount.
\end{remark}

\begin{proof}[Proof of Proposition \ref{prop5.1}]

Corollary~\ref{cor:step1and13} implies
\begin{equation}
\mathcal{N} ( vt - at + s t^{1/3},t) = \frac{(1-v+a)^2}{4} t + \left( \eta_2 + \frac{(1 - v +a) s}{2} \right) t^{1/3} + o( t^{1/3})
\end{equation}
with $\eta_2=-c_2\xi_2$, and
\begin{equation}
\mathcal{N} \left( vt + bt + s t^{1/3}, t \right) = \frac{(1-v-b)^2}{4} t + \left( \eta_1 + \frac{(1 - v-b) s}{2} \right) t^{1/3} + o( t^{1/3}),
\end{equation}
with $\eta_1=-c_1 \xi_1$. The value of $v$ is chosen to satisfy
\begin{equation}
 \frac{(1-v+a)^2}{4} t- \frac{(1-v-b)^2}{4} t=at.
\end{equation}
Using Proposition~\ref{prop:shock1} and collecting terms, we arrive at the statement.

\end{proof}

\begin{prop}\label{prop5.2}
Let $a,b>0$, $2/3< \delta <1$ be fixed real numbers. Consider $M_{+} = \lfloor a t^{\delta} \rfloor$, $M_{-} = \lfloor b t^{\delta} \rfloor$, $v = \frac{b-a}{b+a}$, $r = \frac{a+b}{2}$. We have
\begin{equation}
\lim_{t \to \infty} \frac{\mathfrak{f}^{(1)} (t) - vt - r t^{\delta}}{t^{4/3 - \delta}}\stackrel{d}{=}\frac{2}{a+b}c_3(\xi_1-\xi_2),
\end{equation}
where $c_3= \frac{(1-v^2)^{2/3}}{2^{4/3}}$, and $\xi_1$ and $\xi_2$ are two independent GUE Tracy-Widom distributed random variables. Furthermore,
\begin{equation}
\begin{aligned}
\lim_{t \to \infty} &\left( \frac{\mathfrak{N}^{(1)}_{1} ( v t + r t^{\delta} + s t^{4/3-\delta} ,t) - \mu}{t^{1/3}} ,  \frac{\mathfrak{N}^{(1)}_{2} ( v t + r t^{\delta} + s t^{4/3-\delta} ,t) - a t^{\delta}}{t^{1/3}}  \right)\\
\stackrel{d}{=} &\left( -c_3 \xi_1 , \min\{c_3(\xi_1 - \xi_2) - \tfrac{a+b}{2} s , 0\}\right),
\end{aligned}
\end{equation}
where $\mu=\frac{(1-v)^2}{4}t - \frac{(1-v)(r+b)}{2} t^{\delta} - \frac{(r+b)^2}{4} t^{2 \delta-1} - \frac{(1-v) s}{2} t^{4/3-\delta} - \frac{(r+b) s}{2} t^{1/3}$.
\end{prop}

\begin{remark}
The qualitative behavior in this case is somewhat similar to the one described in Proposition \ref{prop5.1} and Remark \ref{rem:4444}, since the Tracy-Widom fluctuations generated by two blocks remain to be independent. Note, however, that the law of large numbers for the shock position is more delicate here, and that the fluctuations of the second class particle are of nontrivial order $t^{4/3 - \delta}$.
\end{remark}

\begin{proof}[Proof of Proposition \ref{prop5.2}]
Similarly to the previous proposition, we combine Proposition~\ref{prop:shock1}, Corollary~\ref{cor:stepDeltaLarge}, and collect terms, using that $(a+b)(1-v) = 2a$.
\end{proof}

\begin{prop}\label{prop5.3}
Let $a,b>0$, $0< \delta <2/3$ be fixed. Set $M_{+} = \lfloor a t^{\delta} \rfloor$, $M_{-} = \lfloor b t^{\delta} \rfloor$, and $v = \frac{b-a}{b+a}$. Then
\begin{equation}
\lim_{t \to \infty} \frac{\mathfrak{f}^{(1)} (t) - v t}{t^{1- \delta/2}} \stackrel{d}{=} G \left( 0, \frac{4ab}{(a+b)^3} \right)
\end{equation}
and
\begin{equation}
\lim_{t \to \infty} \frac{\mathfrak{N}^{(1)}_{2} ( v t + s t^{1 -\delta/2} ,t) - a t^{\delta}}{t^{\delta/2}} \stackrel{d}{=}
\min \left\{ \frac{s(a+b)}{2} + \sqrt{\frac{ab}{a+b}}G(0,1), 0 \right\},
\end{equation}
where the convergence is in distribution, and $G(0,\sigma^2)$ stands for the Gaussian random variable with zero mean and variance $\sigma^2$.
\end{prop}

\begin{remark}
Compared to the previous two cases, the size of the right block is ``smaller''. This turns out to imply that the shock created by it can be found in a larger region (of order $t^{1-\delta/2}$). In particular, we see that if $\delta \to 0$, then the shock fluctuates on the scale $t$, which corresponds to known results about the order of fluctuations of a second class particle for fixed (not depending on time) initial configurations, see \cite{FK95}, \cite{BB19}. Note that the precise limiting distribution of the second class particle in the $\delta=0$ case is very sensitive to an initial configuration, and it is unclear to us whether the result about the double limit $\delta \to 0$, $t \to \infty$ can give information about the  limiting distributions in the $\delta=0$ case.
\end{remark}

\begin{proof}[Proof of Proposition \ref{prop5.3}]
Let $x = vt + s t^{1 - \delta/2}$, for $s \in \R$. By Corollary~\ref{cor:stepDeltaSmall} we have
\begin{multline}
\mathcal{N} \left( vt + s t^{1 - \delta/2} - a t^{\delta}, t \right) - \mathcal{N} ( vt + s t^{1 - \delta/2} + b t^{\delta} ,t) = t^{\delta} \frac{(1-v)(a+b)}{2} + t^{\delta/2} \frac{s(a+b)}{2} \\ - t^{\delta/2} G(0,1) \frac{\sqrt{(a+b)(1-v^2)}}{2}.
\end{multline}
Plugging the expression for $v$ and using Proposition~\ref{prop:shock1}, we arrive at the statement.
\end{proof}

\subsection{Two colliding GUE-GUE shocks}\label{sec:theorems2shock}
In this section we analyze the behavior of the second class particle in the two colliding shocks. Proposition~\ref{prop:shock2} and results from Section~\ref{sec:StepAsymp} allow to find the behavior of the second class particle, and also the height function of a process with three classes by a direct computation for all four possible scalings from Section~\ref{sec:theorems1shock}. We restrict ourselves with stating the results for the two scalings as the remaining can be computed analogously.

Recall the process $\eta_t^{(2)} (z)$ and denote the position of the second class particle in it as $\mathfrak{f}^{(2)}(t)$.

\begin{prop}\label{prop:2ndClass2shocksTW}
Let us fix $0 < m \le n <1$ with $n<2\sqrt{m}-m$. let $M = \lfloor m t \rfloor$, $N = \lfloor n t \rfloor$, and let $\xi_1, \xi_2, \xi_3$ be three independent GUE Tracy-Widom distributions. One has the following cases.

(a) If $2 (m - n) + (m+n)^2 >0$, set  $v := \frac{m -n + (m+n)^2/2}{m+n}>0$. We have
\begin{equation}
\lim_{t \to \infty} \frac{\mathfrak{f}^{(2)} (t) - v t}{t^{1/3}} \stackrel{d}{=} \frac{(1-v^2)^{2/3}}{2^{1/3} (m+n)} \xi_2 - \frac{(1-(v-m-n)^2)^{2/3}}{2^{1/3} (m+n)} \xi_1.
\end{equation}

(b) If $2 (m - n) + (m+n)^2 = 0$, which is equivalent to $n = 1 - m - \sqrt{1- 4m}$ and requires $m<1/4$, one has
\begin{equation}
\lim_{t \to \infty} \Pb \left( \frac{\mathfrak{f}^{(2)} (t) }{t^{1/3}} \ge s \right)\stackrel{d}{=}
\Pb \left( c_2\xi_2-c_1\xi_1-a_1+\max\{c_1 \xi_3- c_2\xi_2-a_1,0\} \geq 0 \right),
\end{equation}
where $c_1=\frac{(1-(m+n)^2)^{2/3}}{2^{4/3}}$, $c_2=\frac{1}{2^{4/3}}$, $a_1=\frac{(m+n)s}{2}$. Furthermore,
\begin{equation}\label{eq3.13b}
\begin{aligned}
\lim_{t \to \infty} &\left( \frac{\mathfrak{N}^{(2)}_{1} (s t^{1/3} ,t) - \frac{(1-m-n)^2}{4} t }{t^{1/3}} , \frac{\mathfrak{N}^{(2)}_{2} (s t^{1/3} ,t) - m t}{t^{1/3}}, \frac{\mathfrak{N}^{(2)}_{3} (s t^{1/3} ,t) - n t}{t^{1/3}}  \right) \\
\stackrel{d}{=} &\big( -a_2 - c_1 \xi_3 , \min\{c_1  \xi_3 - c_2 \xi_2 - a_1 , 0\}, \\ &\quad\min\big\{c_2 \xi_2 - c_1 \xi_1-a_1 + \max\{c_1 \xi_3 - c_2 \xi_2 - a_1 , 0\}, 0\big\}  \big),
\end{aligned}
\end{equation}
where $a_2=\frac{(1-m-n)s}{2}$.

(c) If $2 (m - n) + (m+n)^2 < 0$ with $m<1/4$, then
\begin{equation}
\lim_{t \to \infty} \frac{\mathfrak{f}^{(2)} (t)}{t^{1/3}} \stackrel{d}{=}  \frac{(1-(m+n)^2)^{2/3}}{2^{4/3} (m+n)}\left( \xi_3 - \xi_1 \right) .
\end{equation}

\end{prop}

\begin{remark}
Three cases demonstrate the evolution of one second class particle in time. Let us describe the meaning of our assumptions and limiting theorems for the process with one second class particle.
The condition $n\leq 1$ implies that at time $t$ the first particle to the right of the second class particle can move, while $m\leq 1$ implies that particles starting to the left of the second class particle reached its starting position. We have three blocks which generate two shocks. If $m > n$, then the shocks will never meet during the evolution, so we assume $m \le n$. We also need to assume $m+n<1$ in order to guarantee that the shocks meet before time $t$.

Case (a) corresponds to the beginning of the evolution under our assumptions: the second class particle moves to the left with constant speed, and its fluctuations are given as the difference of two independent GUE Tracy-Widom distributions. The particle does not feel the left-most (half-infinite) block of particles, so this is the one shock case produced by the second and the third block.

Case (b) is the critical one, it corresponds to the colliding of the two shocks. The second class particle is in the neighborhood of $0$, and its $t^{1/3}$ fluctuations depend on three independent GUE Tracy-Widom distributions, which are generated by three blocks of particles.

Finally, case (c) is again the one shock case produced by the left-most and the right-most blocks of particles; the middle block of particles completely overtakes the second class particle at this moment. Thus the fluctuations are again given by a difference of two independent GUE Tracy-Widom distributions.

We also study the process with three classes of particles corresponding to three initial blocks. We give the statement for height functions only in case (b), see \eqref{eq3.13b}, as it is the most interesting, while simpler statements for cases (a) and (c) are omitted. The result shows how many second and third class particles are overtaken by first class particles in this critical regime.
\end{remark}

\begin{remark}
A similar result but for a different observable has been obtained with different methods in~\cite{FN19}. In that work, one analyzes the fluctuations of tagged particles in the case where two shocks with GOE Tracy-Widom distributed fluctuations merge.
\end{remark}

\begin{proof}[Proof of Proposition~\ref{prop:2ndClass2shocksTW}]
We combine Proposition~\ref{prop:shock2} and Corollary~\ref{cor:step1and13} in order to compute the desired distributions. Case (a) appears
when in the main formula \eqref{eq:SchockExact2} 0 is surely the largest term inside the maximum. Case (c) appears when inside the maximum $h(x+M+N,t)-h(x,t)-M$ is always the largest term, while case (b) is a critical one. The rest is a direct computation.
\end{proof}

In the next proposition we restrict ourselves with the case when $M=N$ in order to somewhat simplify expressions. The statement for $M \ne N$ can be done by analogous computations.

\begin{prop}
\label{prop:2ndClass2shocksAiry}
Let us fix $m \ge 0$, and let $M = N = \lfloor m t^{2/3} \rfloor$. Let ${\cal A}(s)=2^{-4/3}{\cal A}_2(2^{-1/3} s)$, where ${\cal A}_2$ is the Airy$_2$ process.
Then we have
\begin{equation}
\begin{aligned}
\lim_{t \to \infty} \Pb \left( \frac{\mathfrak{f}^{(2)} (t)}{t^{2/3}} \ge s \right)
= &\Pb \Big( m(m-s) - \mathcal A (s-2m) +\mathcal A (s) \\ &+ \max \left\{-m (m+s) -\mathcal A (s) +\mathcal A (2m+s),0 \right\} \ge 0 \Big).
\end{aligned}
\end{equation}
and
\begin{equation}
\begin{aligned}
\lim_{t \to \infty}& \left( \frac{\mathfrak{N}^{(2)}_{1} (s t^{1/3} ,t) - t/4 + \frac{2m+s}{2} t^{2/3} }{t^{1/3}} , \frac{\mathfrak{N}^{(2)}_{2} (s t^{1/3} ,t) - m t^{2/3}}{t^{1/3}}, \frac{\mathfrak{N}^{(2)}_{3} (s t^{1/3} ,t) - m t^{2/3}}{t^{1/3}}  \right) \\
\stackrel{d}{=}&\bigg( \frac{(2m+s)^2}{4} - \mathcal A (2m+s), \min \left\{ - m (m+s) -\mathcal A (s)+ \mathcal A(2m+s), 0 \right\},
 \\ &\quad\min \big\{m(m-s) -\mathcal A(s-2m) + \mathcal A(s) \\& \quad\quad + \max\left\{- m (m+s) -\mathcal A(s) + \mathcal A(2m+s),0  \right\}, 0 \big\} \bigg).
\end{aligned}
\end{equation}

\begin{remark}
In this Proposition we study the colliding of three shocks in the KPZ scaling, which leads to the presence of three sections of Airy process, in the limit, from the three initial blocks of particles.
\end{remark}

\begin{proof}[Proof of Proposition \ref{prop:2ndClass2shocksAiry}]
Proposition~\ref{prop:TasepAiry} implies that
\begin{equation}
\begin{aligned}
&\mathcal{N} \left( \lfloor (2m+s) t^{2/3} \rfloor, t \right) = \frac{t}{4} - \frac{2m+s}{2} t^{2/3} + \frac{(2m+s)^2}{4} t^{1/3} - \mathcal A(2m+s) t^{1/3} + o( t^{1/3}),\\
&\mathcal{N} \left( \lfloor s t^{2/3} \rfloor, t \right) = \frac{t}{4} - \frac{s}{2} t^{2/3} + \frac{s^2}{4} t^{1/3} - \mathcal A(s) t^{1/3} + o( t^{1/3}),\\
&\mathcal{N} \left( \lfloor (s-2m) t^{2/3} \rfloor, t \right) = \frac{t}{4} - \frac{s-2m}{2} t^{2/3} + \frac{(s-2m)^2}{4} t^{1/3} - \mathcal A(s-2m) t^{1/3} + o( t^{1/3}),
\end{aligned}
\end{equation}
Using these expressions and Proposition~\ref{prop:shock2}, we arrive at the expressions from the statement.

\end{proof}

\end{prop}

\section{Backwards geodesics and localization around the characteristic line}\label{sec:Geodesics}
Let us consider TASEP with jump rate $1$ to the right. Let $\eta_j(t)$ be the occupation variable of site $j\in\Z$ at time $t\in\R_+$. Define by $J(t)$ the number of particles which jumped from site $0$ to site $1$ during the time interval $[0,t]$. Then define the height function
\begin{equation}\label{eq3.1}
h(x,t)=\left\{
\begin{array}{ll}
2 J(t)+\sum_{y=1}^{x} (1-2\eta_y(t)),&\textrm{for }x\geq 1,\\
2 J(t),&\textrm{for }x=0,\\
2 J(t)-\sum_{y=x+1}^0(1-2\eta_y(t)),&\textrm{for }x\leq -1.
\end{array}
\right.
\end{equation}
This height function is related to another quantity in the case that initially to the right of the origin there is a finite number of particle. Let ${\cal N}(x,t)$ be the number of particles weakly to the right of $x$ at time $t$. Then
\begin{equation}
h(x,t) = 2 \mathcal N (x,t)+x.
\end{equation}

Consider here the graphical construction of TASEP~\cite{Har78}. At each site there is a Poisson process with intensity $1$, all independent of each other. If there is an event of the Poisson process at position $x$ and time $t$, then if there is a particle at site $x$, it tries to jump to site $x+1$ and this jump occurs only if site $x+1$ is empty. In particular, thus both sites $x$ and $x+1$ are occupied, then nothing happens. In terms of height function we have the following dynamics:
\begin{equation}
  h(x,t)\to h(x,t)+2 \textrm{ iff }x\textrm{ it is a local minimum at time }t.
\end{equation}

\subsection{Concatenation property of the height function}
In~\cite{Fer18} TASEP was described in terms of labelled particle positions and a concatenation property was derived from a backwards space-time path. The same equation was derived before in Lemma~2.1 of~\cite{Sep98c}. Since the quantity we want to analyze is the height function and not the position of a given particle, let us first derive a concatenation property for the height function directly. The construction has analogies and differences with respect to the construction of~\cite{Fer18}. In particular, the backwards path is not unique.

For any time $\tau \in [0,t]$, define $h^{\rm step}_{y,\tau}(x,t)$ to be the height function of the step initial condition at $(x,t)$ with initial condition at time $\tau$ given by $h^{\rm step}_{y,\tau}(x,\tau)=|x-y|$. Then, for all $y\in\Z$, we have
\begin{equation}
h(y,\tau)\leq \hat h(y,\tau):=h(x,\tau)+h^{\rm step}_{y,\tau}(x,\tau).
\end{equation}
At time $\tau$ we have two height profiles, $\hat h(y,\tau)$ and $h(y,\tau)$, which agree at site $x$. We couple the evolution of these two processes from time $\tau$ to time $t$ by the graphical construction (the basic coupling for TASEP). Then, by monotonicity we have $h(y,t)\leq \hat h(y,t)$ for all $y\in\Z$, which gives
\begin{equation}\label{eq3.4a}
h(x,t)\leq \min_{y\in\Z}\{h(y,\tau)+h^{\rm step}_{y,\tau}(x,t)\}.
\end{equation}
Below we want to describe a backwards path $s\mapsto x(s)$ such that if $x=x(t)$, then the inequality in \eqref{eq3.4a} becomes an equality for $y=x(\tau)$.

Let us think what can happen when a Poisson process event occur, say at $(y,s)$. Let $s_-$ denote an infinitesimal time before $s$, so that $h(y,s_-)=\lim_{\e\downarrow 0} h(y,s-\e)$. We have the following possibilities:
\begin{itemize}
\item[(a)] nothing happens if $h(y,s_-)$ is a local maximum,
\item[(b)] $h(y,s)=h(y,s_-)+2$ if $h(y,s_-)$ is a local minimum,
\item[(c)] if $h(y,s_-)$ is neither a local maximum or a local minimum, then the trial is suppressed.
\end{itemize}

\begin{defin}\label{DefBackwardsPath}
Let us define the following backwards process $\tau\mapsto (x(\tau),\tau)$, with $\tau$ running backwards from $t$ to $0$, as follows. In case (a) and (b) nothing happens, i.e., $x(s)=x(s_-)$. However, if we are in case (c), where a trial is suppressed, then
\begin{itemize}
\item[(c1)] if $h(x(s),s)=h(x(s)+1,s)+1$, then we set $x(s_-)=x(s)+1$, i.e., $h(x(s_-),s_-)=h(x(s)+1,s)$,
\item[(c2)] if $h(x(s),s)=h(x(s)-1,s)+1$, then we set $x(s_-)=x(s)-1$, i.e., $h(x(s_-),s_-)=h(x(s)-1,s)$.
\end{itemize}
\end{defin}

Given any time $\tau\in[0,t]$ and the position of the process $x(\tau)$, let us define a the height function
\begin{equation}
x\mapsto \tilde h(x,s)\textrm{ for }s\in[\tau,t]
\end{equation}
as follows. At time $s=\tau$, we set
\begin{equation}\label{eq3.3}
\tilde h(x,t)= h(x(\tau),\tau)+ h^{\rm step}_{x(\tau),\tau}(x,t).
\end{equation}
Thus at time $\tau$ we have two height profiles, $\tilde h(x,\tau)$ and $h(x,\tau)$, which agree at site $x=x(\tau)$, and for any $y\in\Z$, $\tilde h(y,\tau)\geq h(y,\tau)$. We let evolve the two processes by the basic coupling. As before, by monotonicity we get $\tilde h(x,t)\geq h(x,t)$ for all $x\in\Z$.

The following proposition tells us that for TASEP the inequality \eqref{eq3.4a} is actually an equality if we choose $y=x(\tau)$. Since the height function at time $t$ can be given in terms of some height function at any intermediate time $\tau$ plus the increment of a step-initial condition starting from time $\tau$, we call this \emph{concatenation property} in analogy to the name used in the LPP framework.
\begin{prop}\label{PropConcatenation}
For any $\tau\in [0,t]$, it holds
\begin{equation}\label{eqConcatenation}
h(x,t)= \min_{y\in\Z}\{h(y,\tau)+h^{\rm step}_{y,\tau}(x,t)\},
\end{equation}
where the equality is obtained \emph{at least} for $y=x(\tau)$, the position of the backwards path starting from position $x$ at time $t$.
\end{prop}

\begin{proof}
We wish to prove that for $x=x_0=x(t)$, $\tilde h(x,t)=h(x,t)$, i.e., the inequality in \eqref{eq3.4a} is an equality at $x=x_0$, the starting site of our backwards path.

At time $\tau$ we have $h(x(\tau),\tau)=\tilde h(x(\tau),\tau)$. Let $\tilde t$ be the first time after $\tau$ where $x(\tilde t)\neq x(\tau)$. Then we claim that
\begin{equation}\label{eq3.5}
h(x(s),s)=\tilde h(x(s),s),\textrm{ for all }s\in[\tau,\tilde t).
\end{equation}
By assumption $x(s)=x(\tau)$ for all $s\in[\tau,\tilde t)$. Since $h$ and $\tilde h$ are coupled, the only way that \eqref{eq3.5} stops being satisfied is that there is some time $u\in [\tau,\tilde t)$ such that $h(x(\tau),u)<\tilde h(x(\tau),u)$, i.e., at time $u$ a trial at position $x(\tau)$ is suppressed for $h$, but not for $\tilde h$. This implies that at time $u$, $x(u)\neq x(\tau)$, which is a contradiction.

Let us see what happens at time $\tilde t$. There are two symmetric cases, depending on whether the trial is suppressed at a decreasing or increasing part of the height function. The arguments are completely symmetric, so we give the details only in one case. Consider the case that $x(\tilde t_-)=x(\tilde t)+1$, see Figure~\ref{FigLocalConfig}. Then we have
\begin{figure}
\begin{center}
\psfrag{x}[c]{$x(\tilde t)$}
\psfrag{y}[c]{$x(\tilde t_-)$}
\includegraphics[height=3cm]{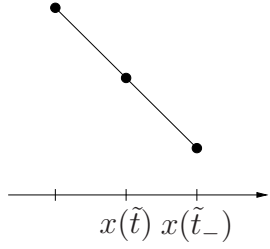}
\caption{Local configuration of the height function $h$ in the case $x(\tilde t)=x(\tilde t_-)-1$.}
\label{FigLocalConfig}
\end{center}
\end{figure}
\begin{equation}\label{eq3.6}
\begin{aligned}
h(x(\tilde t),\tilde t) &= h(x(\tilde t_-),\tilde t)+1,\\
h(x(\tilde t_-),\tilde t_-)&=\tilde h(x(\tilde t_-),\tilde t_-),\\
h(x(\tilde t),\tilde t)&\leq \tilde h(x(\tilde t),\tilde t),
\end{aligned}
\end{equation}
where the first equality reflects the fact that there was a suppressed jump at time $\tilde t$, the second follows by \eqref{eq3.5} since $\tilde t_-<\tilde t$ and the third holds generically, see \eqref{eq3.3}.

As a consequence, just before the suppressed trial for $h$, the height function $\tilde h$ and $h$ are equal on positions $x(\tilde t)$ and $x(\tilde t)+1$. Thus the suppressed trial for $h$ is also a suppressed trial for $\tilde h$ and thus we have
\begin{equation}
h(x(\tilde t),\tilde t)=\tilde h(x(\tilde t),\tilde t)
\end{equation}
as well.

Repeating this argument at each time when $x(u)$ has a jump one gets
\begin{equation}
h(x(u),u)=\tilde h(x(u),u)
\end{equation}
for all $u\in [\tau,t]$. In particular, $h(x(t),t)=\tilde h(x(t),t)$, which is the claimed result.
\end{proof}

As a consequence of \eqref{eqConcatenation}, there exists \emph{at least} one path going from time $t$ to time $0$ such that for any time $\tau$ \eqref{eqConcatenation} holds. However this path is in general not unique.
\begin{defin}
A path $\{(x(\tau),\tau),\tau:t\to 0\}$ with $x(t)=x$ and satisfying \eqref{eqConcatenation} for all $\tau\in[0,t]$ is called backwards geodesic from time $t$ and position $x$ to time $0$.
\end{defin}

\begin{remark}\label{remGeodesics}
For TASEP with step initial condition, i.e., $h(x,0)=|x|$, there exists a backwards geodesic ending at $x(0)=0$. Indeed, consider the backwards path constructed above. If it ends at $x(0)=0$, then we are done. In case that $x(0)=x_0\neq 0$, then let $\e>0$ such that no Poisson point events between $x_0$ and $0$ occur in the time interval $[0,\e]$. Then, we have $x(\e)=x_0$. Define a new path $\tilde x$ obtained by modifying the path $x$ only during $s\in [0,\e]$ by setting $\tilde x(s)=\lfloor x_0 s/\e\rfloor$. Then $\tilde x$ is also a geodesic and it satisfies $\tilde x(0)=0$.
\end{remark}

What information do we get from the location of a geodesic? Consider a geodesic ending at $(x,t)$. At any time, we can reset the configuration to a step initial condition at the geodesic and the height function at time $(x,t)$ will not be effected. In particular, by monotonicity, we can reset the configuration to anything between the actual one and step initial condition. Thus if we know that with high probability the geodesic is in a deterministic region $\cal D$, then with high probability $h(x,t)$ is depending on the randomness in $\cal D$ only.

Let us formalize it slightly more. Consider a deterministic space-time region $\cal D$ and let the event $\Omega_{\rm loc}=\{(x(\tau),\tau)_{0\leq\tau\leq t}\in {\cal D}\}$, where $\tau\mapsto x(\tau)$ is a geodesic ending at $(x,t)$. On $\Omega_{\rm loc}$, we replace at any time the system to have slope $-1$ to the left of $\cal D$ and slope $+1$ to the right of $\cal D$. Equivalently, we can think of replacing the Poisson processes outside $\cal D$ with Poisson processes with infinite rate, which leads to fully filled particles to the left and fully empty to the right of $\cal D$. Thus, we can write
\begin{equation}
h(x,t)=h(x,t)\Id_{\Omega_{\rm loc}}+h(x,t)\Id_{\Omega_{\rm loc}^c},
\end{equation}
where $h(x,t)\Id_{\Omega_{\rm loc}}$ is independent of the randomness outside $\cal D$. In particular,
\begin{equation}\label{eq3.14b}
\Pb(h(x,t)\neq h(x,t)\Id_{\Omega_{\rm loc}})\geq \Pb(\Omega_{\rm loc}).
\end{equation}

\subsection{Decorrelation of the small time randomness}
Here we prove a result which is very close to the slow-decorrelation result discussed in Section 3.1 of~\cite{CFP10b}. Consider step-initial condition. Then the evolution over a time $\tau_t=o(t)$ is irrelevant for the fluctuations of the height function at position $\alpha t$ at time $t$. The latter are asymptotically the same as the fluctuations generated by step-initial condition from time $\tau_t$ to time $t$.
\begin{prop}\label{PropSlowDec}
Let $h(x,0)=|x|$ and $\alpha\in (-1,1)$. Consider $\{\tau_t\}$ a sequence of times such that $\tau_t/t\to 0$.
We have the following:
\begin{equation}
\forall \e>0, \, \lim_{t\to \infty} \Pb(|h(\alpha t,t)-h^{\rm step}_{\alpha \tau_t,\tau_t}(\alpha t,t)-\tfrac12 (1+\alpha^2)\tau_t|\geq \e t^{1/3})=0.
\end{equation}
\end{prop}
\begin{proof}
We use the convergence in distribution of the following rescaled random variables
\begin{equation}
\begin{aligned}
h^{(1)}_t&:=\frac{h(\alpha t,t)-\tfrac12(1+\alpha^2)t}{-2^{-1/3}(1-\alpha^2)^{2/3}t^{1/3}}\stackrel{t\to\infty}{\Longrightarrow}F_{\rm GUE},\\
h^{(2)}_t&:=\frac{h^{\rm step}_{\alpha \tau_t,\tau_t}(\alpha t,t)-\tfrac12(1+\alpha^2)(t-\tau_t)}{-2^{-1/3}(1-\alpha^2)^{2/3}t^{1/3}}\stackrel{t\to\infty}{\Longrightarrow}F_{\rm GUE},\\
h^{(3)}_t&:=\frac{h(\alpha \tau_t,\tau_t)-\tfrac12(1+\alpha^2)\tau_t}{-2^{-1/3}(1-\alpha^2)^{2/3}\tau_t^{1/3}}\stackrel{t\to\infty}{\Longrightarrow}F_{\rm GUE}.
\end{aligned}
\end{equation}
The concatenation property \eqref{eqConcatenation} gives
\begin{equation}
h^{(1)}_t\leq h^{(2)}_t+(\tau_t/t)^{1/3} h^{(3)}_t,
\end{equation}
which implies that
\begin{equation}
h^{(1)}_t= h^{(2)}_t+(\tau_t/t)^{1/3} h^{(3)}_t-X_t,
\end{equation}
where $X_t\geq 0$ is a random variable. Since $\tau_t/t\to 0$ as $t\to\infty$, we have that $(\tau_t/t)^{1/3} h^{(3)}_t$ converges to $0$, which implies that therefore $h^{(1)}_t$ and $h^{(2)}_t+(\tau_t/t)^{1/3} h^{(3)}_t$ converge to the same distribution. Then Lemma~\ref{LemCorwin} implies that $X_t$ converges to $0$ in probability. The second statement of Lemma~\ref{LemCorwin} then gives the claimed statement.
\end{proof}

\subsection{The comparison lemma}
We are interested in controlling the increment of the height function. In particular, we want to apply the result for distances which are $o(t^{2/3})$, i.e., less than the natural KPZ correlation scale. As the limiting process, the Airy$_2$ process in our case, is locally Brownian, it is natural to try to estimate the variation of the increments with respect to a stationary situation. This can be obtained by adapting the idea previously developed in the last passage percolation framework~\cite{Pim17,CP15b}.

Consider the evolution of two coupled height functions, starting with initial profiles $h_1(x,0)$ and $h_2(x,0)$. Consider two positions $x<y$. For each of the height functions and each of the two end positions, construct the backwards geodesics. We call them
\begin{equation}
\pi_{k,x}=\{(x_k(s),s),s:t\to 0\},\quad \pi_{k,y}=\{(y_k(s),s),s: t\to 0\},\quad k=1,2,
\end{equation}
where with the notation $s:t\to 0$ we want to stress that the time goes backwards, from $t$ to $0$.
We have the following comparison lemma.
\begin{figure}
\begin{center}
\psfrag{t}[c]{$t$}
\psfrag{x}[c]{$x$}
\psfrag{y}[c]{$y$}
\psfrag{1x}[l]{$\pi_{1,x}$}
\psfrag{1y}[l]{$\pi_{1,y}$}
\psfrag{2x}[r]{$\pi_{2,x}$}
\psfrag{2y}[r]{$\pi_{2,y}$}
\includegraphics[height=4cm]{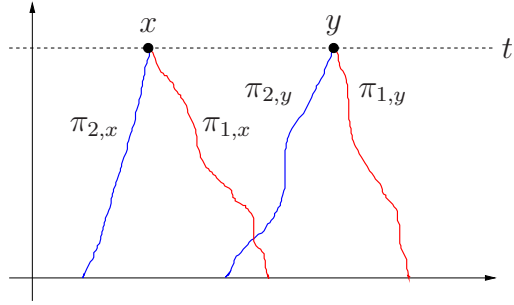}
\caption{Backwards geodesics with intersection of $\pi_{1,x}$ and $\pi_{2,y}$.}
\label{FigComparison}
\end{center}
\end{figure}

\begin{lem}\label{lemComparison}
Let $x<y$. On the event $\{\pi_{1,x}\cap\pi_{2,y}\neq \emptyset\}$, we have
\begin{equation}\label{eq3.12}
h_2(y,t)-h_2(x,t) \geq h_1(y,t)-h_1(x,t).
\end{equation}
Similarly, on the event $\{\pi_{1,y}\cap\pi_{2,x}\neq \emptyset\}$, we have
\begin{equation}\label{eq3.13}
h_2(y,t)-h_2(x,t) \leq h_1(y,t)-h_1(x,t),
\end{equation}
\end{lem}
\begin{proof}
Assume that $\pi_{1,x}\cap\pi_{2,y}=(x_\tau,\tau)\neq\emptyset$. Then,
\begin{equation}\label{eq3.14}
\begin{aligned}
h_2(y,t)&=h_2(x_\tau,\tau)+h^{\rm step}_{x_\tau,\tau}(y,t),\\
h_2(x,t)&\leq h_2(x_\tau,\tau)+h^{\rm step}_{x_\tau,\tau}(x,t),
\end{aligned}
\end{equation}
but also
\begin{equation}\label{eq3.15}
\begin{aligned}
h_1(x,t)&=h_1(x_\tau,\tau)+h^{\rm step}_{x_\tau,\tau}(x,t),\\
h_1(y,t)&\leq h_1(x_\tau,\tau)+h^{\rm step}_{x_\tau,\tau}(y,t).
\end{aligned}
\end{equation}
Combining \eqref{eq3.14} and \eqref{eq3.15} we get
\begin{equation}
  h_1(y,t)-h_1(x,t)\leq h^{\rm step}_{x_\tau,\tau}(y,t)-h^{\rm step}_{x_\tau,\tau}(x,t) \leq h_2(y,t)-h_2(x,t).
\end{equation}
Thus we have proved \eqref{eq3.12}. The proof of \eqref{eq3.13} is analogous.
\end{proof}

\subsection{Localization for the stationary case}
The stationary and translation-invariant measure of TASEP is Bernoulli product measure with parameter $\rho\in[0,1]$, the average density of particle~\cite{Li99}. In order to do comparison between the increments and the stationary increments of the height function, we will apply Lemma~\ref{lemComparison} where we take for $h_1$ and $h_2$ two different stationary initial conditions. The events where there is intersection of the backwards paths can be controlled once we have good estimates on the positions of the location of $\pi_{2,x}$ at time $0$.

\begin{prop}\label{PropLocalStatInitialPoint}
Let us consider stationary TASEP with density $\rho\in (0,1)$. Let $x(t)=(1-2\rho)t$ be the starting position of any backwards geodesic. Then, there exists constants $C,c>0$ independent of $t$ such that
\begin{equation}
\Pb(|x(0)|\geq M t^{2/3})\leq C e^{-c M^2}
\end{equation}
uniformly for all $t$ large enough.
\end{prop}
\begin{proof}
First we consider a related problem, namely let $x(t)=(1-2\rho_0)t=\tilde x$ and $\rho_\pm = \rho_0 \pm \tilde \kappa t^{-1/3}$. Since in the proof we are going to consider different densities for stationary TASEP, we use the notation $\Pb_\varrho$ to denote the probability law for stationary initial condition with density $\varrho$.

Consider the any backwards geodesic in the stationary case with density $\rho_+$ and look at the quantity $\Pb_{\rho_+}(x(0)>0)$. If $\min_{y>0}\{h(y,0)+h^{\rm step}_{y,0}(\tilde x,t)\}$ is strictly smaller than $\min_{y\leq 0}\{h(y,0)+h^{\rm step}_{y,0}(\tilde x,t)\}$, then it implies that any backwards geodesic is on $\{y>0\}$ at time $0$, and viceversa. In particular, this implies that $x(0)>0$. Therefore we have
\begin{equation}\label{eq4.26}
\Pb_{\rho_+}(x(0)>0) \geq \Pb_{\rho_+}\Big(\min_{y\leq 0}\{h(y,0)+h^{\rm step}_{y,0}(\tilde x,t)\} >\min_{y> 0}\{h(y,0)+h^{\rm step}_{y,0}(\tilde x,t)\} \Big).
\end{equation}
We would have equality if the geodesic would be unique, but due to the non-uniqueness, there is still the possibility that the minimum is attaint simultaneously on both sides of the origin. Then, for any choice of $A\in\Z$,
\begin{equation}
\begin{aligned}
&\eqref{eq4.26}\geq \Pb_{\rho_+}\Big(\min_{y\leq0}\{h(y,0)+h^{\rm step}_{y,0}(\tilde x,t)\} > A \geq \min_{y> 0}\{h(y,0)+h^{\rm step}_{y,0}(\tilde x,t)\} \Big)\\
&\geq 1-\Pb_{\rho_+}\Big(\min_{y\leq0}\{h(y,0)+h^{\rm step}_{y,0}(\tilde x,t)\} \leq A\Big) - \Pb_{\rho_+}\Big(\min_{y> 0}\{h(y,0)+h^{\rm step}_{y,0}(\tilde x,t)\}>A \Big).
\end{aligned}
\end{equation}
These probabilities for an appropriate choice of $A$ have been already bounded in~\cite{FO17} in terms of the last passage percolation model. Denote by $L^{\rho_+}$ the stationary LPP passage time, $L^{\rho_+}_\vert$ the LPP passage time restricted to paths whose first step is from $(0,0)$ to $(0,1)$, and similarly $L^{\rho_+}_{-}$ the LPP passage time restricted to paths whose first step is from $(0,0)$ to $(1,0)$ (see~\cite{FO17} for more details). Then, by the well-known link between height function and LPP we have
\begin{equation}
\begin{aligned}
\Pb_{\rho_+}\Big(\min_{y<0}\{h(y,0)+h^{\rm step}_{y,0}(\tilde x,t)\} \leq A\Big) &= \Pb\Big(L^{\rho_+}_\vert(\tfrac12 (A+\tilde x),\tfrac12 (A-\tilde x))> t\Big),\\
\Pb_{\rho_+}\Big(\min_{y\geq 0}\{h(y,0)+h^{\rm step}_{y,0}(\tilde x,t)\} > A\Big)& = \Pb\Big(L^{\rho_+}_{-}(\tfrac12 (A+\tilde x),\tfrac12 (A-\tilde x))\leq t\Big).
\end{aligned}
\end{equation}
The LPP goes from the origin to the point $(\gamma^2 n,n)$. Let us set $\chi=\rho_0(1-\rho_0)$. The quantities $\tilde x$ and $t$ being given, with the choice
\begin{equation}
A=(1-2\chi)t-\tilde \kappa^2 t^{1/3}
\end{equation}
we get
\begin{equation}
\begin{aligned}
n&=\rho_0^2 t-\tfrac12 \tilde \kappa^2 t^{1/3},\\
\gamma&=\frac{1-\rho_0}{\rho_0}+\frac{1-2\rho_0}{4\rho_0^2\chi}\tilde\kappa^2 t^{-2/3}+\Or(t^{-4/3}).
\end{aligned}
\end{equation}
Furthermore, setting $\rho_+=\frac{1}{1+\gamma}+\kappa n^{-1/3}$ as in~\cite{FO17}, we get
\begin{equation}
\kappa = \tilde \kappa \rho_0^{2/3}+\frac{1-2\rho_0}{4\chi^2}\tilde\kappa^3+\Or(t^{-1/3})
\end{equation}
and $x$ in Lemma~3.3 of~\cite{FO17} is indeed given by $t+\Or(1)$ (where the $\Or(1)$ plays no role in the asymptotics). Then (3.8) and Lemma~3.3 of~\cite{FO17} gives the following: for any $t$ large enough,
\begin{equation}
\begin{aligned}
\Pb_{\rho_+}\Big(\min_{y<0}\{h(y,0)+h^{\rm step}_{y,0}(\tilde x,t)\} \leq A\Big) &\leq C e^{-c\tilde\kappa^2},\\
\Pb_{\rho_+}\Big(\min_{y\geq 0}\{h(y,0)+h^{\rm step}_{y,0}(\tilde x,t)\} > A\Big)& \leq C e^{-c\tilde\kappa^3},
\end{aligned}
\end{equation}
for some constants $C,c$. These constants do not depend on $\tilde\kappa$ and can be taken uniformly for $\rho_0$ and $\rho_+$ in a bounded set away from $0$ and $1$.

We have thus proven that
\begin{equation}
\Pb_{\rho_+}(x(0)<0)\leq \tilde C e^{-c\tilde\kappa^2}
\end{equation}
for some constant $\tilde C$. This means that for density $\rho_+$, $x(0)-x(t)\geq (2\rho-1)t$ with high probability. But
\begin{equation}
(2\rho-1)t = (2\rho_+-1)t -2\tilde \kappa t^{2/3},
\end{equation}
thus by translation invariance of the stationary case, for density $\rho_+$, we have that
\begin{equation}
\Pb_{\rho_+}(x(0)<-2\tilde\kappa t^{2/3}| x(t)=(1-2\rho_+)t) \leq \tilde C e^{-c\tilde\kappa^2}.
\end{equation}
Taking $\rho_0=\rho-\tilde\kappa t^{-1/3}$, since the constants $\tilde C,c$ are uniform for densities in a bounded set away from $(0,1)$, we get $\rho_+=\rho$ and thus
\begin{equation}
\Pb_{\rho}(x(0)<-2\tilde\kappa t^{2/3}| x(t)=(1-2\rho)t) \leq \tilde C e^{-c\tilde\kappa^2}.
\end{equation}

Similarly, for density $\rho_-$, we get
\begin{equation}
\Pb_{\rho_-}(x(0)>2\tilde\kappa t^{2/3}| x(t)=(1-2\rho_-)t) \geq \tilde C e^{-c\tilde\kappa^2}
\end{equation}
and setting $\rho_0=\rho+\tilde\kappa t^{-1/3}$ we have $\rho_-=\rho$. Combining these last two bounds with $M=2\kappa$ we obtain
\begin{equation}
\Pb_\rho(|x(0)|\geq M t^{2/3})\leq C e^{-c M^2}
\end{equation}
for some new constants $C,c>0$.
\end{proof}

\subsection{Localization along characteristics for step initial condition.}
Let us start with a mid-time localization estimate of the geodesic.
\begin{prop}[Mid-time estimate]\label{PropMidTime}
Let $h(x,0)=|x|$ and let $\alpha\in (-1,1)$ be fixed. Let $x(\tau)$ be the any backwards geodesic starting from $x(t)=\alpha t$.
Then, for all $t$ large enough and $u>0$ we have
\begin{equation}
\Pb(|x(t/2)-\alpha t/2|\geq u t^{2/3})\leq C_1 e^{-c_1 u^2}
\end{equation}
for some constants $C_1,c_1>0$.
\end{prop}
\begin{proof}
Let us start estimating $\Pb(x(t/2)-\alpha t/2\geq u t^{2/3})$, as the other bound is obtained similarly.

Let $X_t(u)=\alpha t/2+u t^{2/3}$. If $h(\alpha t,t)<\min_{y\geq X_t(u)}\{h(y,t/2)+h^{\rm step}_{y,t/2}(\alpha t,t)\}$ then necessarily any geodesic is at time $t/2$ to the left of $X_t(u)$. Therefore, for any $S$, we have
\begin{equation}\label{eq3.28}
\begin{aligned}
&\Pb(x(t/2)-\alpha t/2< u t^{2/3})\geq \Pb\Big(h(\alpha t,t)<\min_{y\geq X_t(u)}\{h(y,t/2)+h^{\rm step}_{y,t/2}(\alpha t,t)\}\Big)\\
&\geq \Pb\Big(h(\alpha t,t)\leq S <\min_{y\geq X_t(u)}\{h(y,t/2)+h^{\rm step}_{y,t/2}(\alpha t,t)\}\Big)\\
&\geq 1-\Pb(h(\alpha t,t)>S)-\Pb\Big(\min_{y\geq X_t(u)}\{h(y,t/2)+h^{\rm step}_{y,t/2}(\alpha t,t)\}\leq S\Big).
\end{aligned}
\end{equation}
From the law of large numbers we have
\begin{equation}
\begin{aligned}
h(\alpha t,t)&\simeq \tfrac12(1+\alpha^2)t,\\
h(X_t(v),t/2)&\simeq \tfrac14(1+\alpha^2)t+v\alpha t^{2/3}+ v^2t^{1/3},\\
h^{\rm step}_{X_t(v),t/2}(\alpha t,t)&\simeq \tfrac14(1+\alpha^2)t-v\alpha t^{2/3}+v^2t^{1/3}.
\end{aligned}
\end{equation}
Thus we choose $S=\tfrac12(1+\alpha^2)t+\tfrac12 u^2 t^{1/3}$. From the one-point tail bounds we have that
\begin{equation}
\Pb(h(\alpha t,t)>S)\leq C e^{-c u^2}.
\end{equation}
To estimate the last term in \eqref{eq3.28},  we want to divide the minimum into two minimums, but due to the linear term in the law of large numbers, we need to correct with this accordingly. So, define $f(y)=\alpha (y-\alpha t/2)$ so that $f(X_t(v))=\alpha v t^{2/3}$. Then
\begin{equation}\label{eq3.31}
\begin{aligned}
&\Pb\Big(\min_{y\geq X_t(u)}\{h(y,t/2)+h^{\rm step}_{y,t/2}(\alpha t,t)\}\leq S\Big) \\
&\quad\leq\Pb\Big(\min_{y\geq X_t(u)}\{h(y,t/2)-f(y)\}+\min_{y\geq X_t(u)} \{h^{\rm step}_{y,t/2}(\alpha t,t)+f(y)\}\leq S\Big) \\
&\quad\leq\Pb\Big(\min_{y\geq X_t(u)}\{h(y,t/2)-f(y)\}\leq S/2\Big)+\Pb\Big(\min_{y\geq X_t(u)}\{h^{\rm step}_{y,t/2}(\alpha t,t)+f(y)\}\leq S/2\Big),
\end{aligned}
\end{equation}
where the last inequality comes the fact that $a+b<S$ implies that $a<S/2$ or $b<S/2$.

The bound on the two terms of \eqref{eq3.31} are obtained in a similar way as they are both height functions from step initial conditions.

Choose a small $\delta>0$ and decompose
\begin{equation}\label{eq3.32}
\begin{aligned}
&\Pb\Big(\min_{y\geq X_t(u)}\{h(y,t/2)-f(y)\}\leq S/2\Big) \\
& \quad\leq \sum_{\ell=1}^{t^\delta} \Pb\Big(\min_{\ell u\leq z< (\ell+1)u}\{h(X_t(z),t/2)-f(X_t(z))\}\leq S/2\Big)\\
&\quad +\sum_{z\geq (t^\delta+1)u}\Pb\Big(h(X_t(z),t/2)-f(X_t(z))\leq S/2\Big),
\end{aligned}
\end{equation}
where the minimum and the last sum are of course only for $z$ such that $X_t(z)\in\Z$.

\medskip
\emph{Bound on the first term in \eqref{eq3.32}}. Remark that $h(X_t(\ell u),t/2)-f(X_t(\ell u))\simeq \tfrac14(1+\alpha^2)t+u^2\ell^2 t^{1/3}$. From the one-point bound, we have
\begin{equation}\label{eq3.33}
\Pb(h(X_t(\ell u),t/2)-f(X_t(\ell u))\geq S/2+\tfrac12 \ell^2 u^2t^{1/3})\geq 1- C e^{-c u^2 \ell^2}.
\end{equation}
Below we will prove the bound
\begin{equation}\label{eq3.34}
\begin{aligned}
&\Pb\Big(\min_{\ell u\leq z< (\ell+1)u}\{h(X_t(z),t/2)-h(X_t(\ell u),t/2)+\alpha (z-\ell u) t^{2/3}\}\geq -\tfrac12\ell^2u^2t^{1/3}\Big)\\
&\geq 1- \tilde C e^{- \tilde c \ell^2 u^2}
\end{aligned}
\end{equation}
for constants $\tilde C,\tilde c>0$.
Then, \eqref{eq3.33} and \eqref{eq3.34} gives
\begin{equation}
\Pb\Big(\min_{\ell u\leq z< (\ell+1)u}\{h(X_t(z),t/2)-f(X_t(z))\}\leq S/2\Big)\leq C e^{-c u^2\ell^2}+\tilde C e^{-\tilde c u^2\ell^2},
\end{equation}
from which the first sum in \eqref{eq3.32} is bounded by $C e^{-c u^2}$ for some new constants $C,c>0$.

It remains to bound \eqref{eq3.34}. We do it by comparison with the stationary case. Consider $\rho=\tfrac12(1-\alpha)-\ell u t^{-1/3}$, so that $(1-2\rho)t/2 = X_t(\ell u)$. Define $\rho_+=\rho+\tilde\kappa t^{-1/3}$ and the event $\Omega_{\tilde\kappa}$ where the backwards path of Definition~\ref{DefBackwardsPath} for stationary initial condition with density $\rho_+$ starting from $X_t(\ell u)$ is at time $0$ on $\{y\geq 0\}$. Recall that for step initial condition we can take a backwards geodesic starting from $X_t(z)$ which ends at the origin. By Proposition~\ref{PropLocalStatInitialPoint} we have $\Pb(\Omega_{\tilde\kappa})\geq 1-C e^{-c\tilde\kappa^2}$.

On $\Omega_{\tilde\kappa}$, by the comparison lemma (see Lemma~\ref{lemComparison}), we have
\begin{equation}
\begin{aligned}
&h(X_t(z),t/2)-h(X_t(\ell u),t/2)+\alpha (z-\ell u) t^{2/3}\\
&\geq h_{\rho_+}(X_t(z),t/2)-h_{\rho_+}(X_t(\ell u),t/2)+\alpha(z-\ell u) t^{2/3}\\
&=\sum_{k=1}^{X_t(z)-X_t(\ell u)}(\alpha-1+2 Z_k),
\end{aligned}
\end{equation}
where $Z_k$ are i.i.d.\ random variables with $\Pb(Z_k=1)=\rho_+$ and $\Pb(Z_k=0)=1-\rho_+$. Thus $\E(\alpha-1+2 Z_k)=2(\tilde \kappa-\ell u)$. We choose $\tilde \kappa=\ell u$, so that $W_k=2\E(Z_k)-2 Z_k= 1-\alpha-2Z_k$. Thus we get
\begin{equation}\label{eq3.37}
1-\eqref{eq3.34} \leq \Pb(\Omega^c_{\tilde\kappa}) + \Pb\bigg( \max_{0\leq z\leq u} \sum_{k=1}^{z t^{2/3}} W_k\geq \tfrac12\ell^2u^2t^{1/3}\bigg).
\end{equation}
With the chosen value of $\tilde\kappa$ we have $\Pb(\Omega_{\tilde\kappa})\geq 1-Ce^{-c\ell^2u^2}$. Using the Doob maximum inequality we get
\begin{equation}\label{eq3.38}
\Pb\bigg( \max_{0\leq z\leq u} \sum_{k=1}^{z t^{2/3}} W_k\geq \tfrac12\ell^2u^2t^{1/3}\bigg) \leq \inf_{\lambda>0} \frac{(\E(e^{\lambda W_k}))^{u t^{2/3}}}{e^{\lambda \ell^2 u^2 t^{1/3}/2}}.
\end{equation}
A computation gives, for all $0\leq \delta< 1/3$, $\eqref{eq3.38}\leq e^{-c_0 \ell^4 u^3}$ with $c_0=1/(16(1-\alpha^2))$ holds for all $t$ large enough. Thus \eqref{eq3.34} is proven as well.

\medskip
\emph{Bound on the last term in \eqref{eq3.32}}. We are going to use the one-point estimate \eqref{eqBoundUpper} of the rescaled height function of $h(\beta t,t)$, where the bound has constants uniform for $\beta$ in a bounded set of $(-1,1)$. We just have avoid taking positions corresponding to some $\beta$ going to $\pm 1$.

Consider $z\geq t^\delta u$ such that $X_t(z)\leq \beta_c t/2$ with $\beta_c=\tfrac12(1+\alpha)+u^2 t^{-2/3}/(2(1-\alpha))$. Denote $\beta=\alpha+2 z t^{-1/3}$. Then
\begin{equation}\label{eq3.39}
\Pb\Big(h(X_t(z),t/2)-f(X_t(z))\leq S/2\Big) = \Pb\big(h(\beta t/2,t/2)\leq \tfrac14(1+\beta^2)t-s (t/2)^{1/3}\big),
\end{equation}
with $s=2^{1/3} (z^2-u^2/4)\geq t^{2\delta}$ for all $t$ large enough. Thus for all $t$ large enough, by \eqref{eqBoundUpper} we get $|\eqref{eq3.39}|\leq C e^{-c t^{2\delta}}$ for some constants $C,c>0$.

Finally, for $X_t(z)>\beta_c t/2$, $h(X_t(z),t/2)\geq h(X_t(z),0)$ and $h(X_t(z),0) - f(X_t(z)) \geq S/2$, thus for such $z$, $\eqref{eq3.39}=0$.

Putting these bounds and \eqref{eq3.34} into \eqref{eq3.32} the result is proven.
\end{proof}

\begin{prop}\label{PropLocalization}
Consider TASEP with step initial condition. Let $x=x(t)=\alpha t$ the starting position of the backwards path given in Definition~\ref{DefBackwardsPath}, with $\alpha\in(-1,1)$ fixed. Then, uniformly for all $t$ large enough,
\begin{equation}
\Pb(|x(\tau)-\alpha \tau |\leq u t^{2/3}\textrm{ for all }0\leq \tau\leq t)\geq 1- C_2 e^{-c_2 u^2}
\end{equation}
for some constants $C_2,c_2>0$.
\end{prop}
\begin{proof}
The idea follows the approach of~\cite{BSS14}, once we have the localization of the mid-point backwards path from Proposition~\ref{PropMidTime}.

Let us set $N=\min\{n: 2^{-n} t\leq t^{1/2}\}$. Let us choose $u_1<u_2<u_3<\ldots$ with $u_1=u/10$ and $u_n-u_{n-1}= u_1 2^{-(n-1)/2}$. Define the following events
\begin{equation}
\begin{aligned}
A_n&=\{x(k 2^{-n}t)\leq \alpha k 2^{-n}t+u_n t^{2/3},1\leq k\leq 2^n-1\},\\
B_{n,k}&=\{x(k 2^{-n} t)> \alpha k 2^{-n} t + u_n t^{2/3}\},\quad k=1,\ldots,2^n-1,\\
L&=\{\sup_{x\in [0,1]}|x( (k+x) 2^{-N}t)-x(k 2^{-N} t)-\alpha x 2^{-N}t|\leq \tfrac15 u t^{2/3},0\leq k\leq 2^N-1\},\\
G&=\{x(\tau)\leq \alpha \tau +u t^{2/3}\textrm{ for all }0\leq \tau\leq t\}.
\end{aligned}
\end{equation}
First of all, notice that $A_n^c=\bigcup_{k=1}^{2^n-1} B_{n,k}$. Further, since $\lim_{n\to\infty} u_n<\tfrac45 u$, we have
\begin{equation}
\bigcup_{n=1}^N\bigcup_{k=1}^{2^n-1} (B_{n,k}\cap A_{n-1}) \supseteq \{x(k 2^{-N} t)\geq \alpha k 2^{-N} t+ \tfrac45 u t^{2/3}\textrm{ for some }k=1,\ldots,2^N-1\},
\end{equation}
with $A_0=\Omega$ being the whole outcome space. From this we get
\begin{equation}
G \supseteq \bigg(\bigcup_{n=1}^N\bigcup_{k=1}^{2^n-1} (B_{n,k}\cap A_{n-1}) \bigg)^c\cap L,
\end{equation}
since the first term means that in the sequence of discrete $2^N-1$ points the backwards path $x$ is not farther than $\tfrac45 u t^{2/3}$ to the right of the characteristic and the event $L$ controls the possible excursion between these times. Thus
\begin{equation}\label{eq3.44}
\Pb(G)\leq \Pb(L^c)+\sum_{n=1}^N\sum_{k=1}^{2^n-1} \Pb(B_{n,k}\cap A_{n-1}).
\end{equation}

\emph{Bound on $\Pb(L^c)$.} Since the jumps of $x(\tau)$ are stochastically bounded by a Poisson process with intensity $1$ and Poisson distribution decay faster than exponential, we immediately get $\Pb(L^c)\leq 2^N C e^{-u t^{1/6}}\leq e^{-c_2 u^2}$ for all $t$ large enough (for any choice of $c_2>0$).

\emph{Bound on $\Pb(B_{n,k}\cap A_{n-1})$.} For $k$ even, the two events are incompatible so its probability it $0$. Thus consider odd $k$. We denote by $x_{\rm mid}(\tau)=x(\tau/2)-\alpha \tau/2$ the centered position of the midpoint over a time-span $\tau$, given $x(\tau)=\alpha \tau$. If $A_{n-1}$ holds, then the backwards path at time $t_1=(k-1) 2^{-n} t$ and $t_2=(k+1) 2^{-n} t$ are at to the left of $x_i=\alpha t_i+u_{n-1}t^{2/3}$, $i=1,2$.

Therefore under $A_{n-1}$ we have that $x(t_1)\leq x_1$ and $x(t_2)\leq x_2$. Consider the model with step initial condition from time $t_1$ at position $x_1$ and denote by $\tilde x$ the starting at time $t_2$ from position $x_2$. Both models are coupled by the basic coupling. Let us verify that for $s\in [t_1,t_2]$ we have $x(s)\leq \tilde x(s)$.
The backwards path $x$ (resp.\ $\tilde x$) after time $t_1$ is determined by the evolution of the height function $h$ (resp.\ $\tilde h$) starting with the step initial condition at time $t_1$ and position $x(t_1)$ (resp.\ $x_1$). Then, since initially there are particles for $\tilde h$ whenever there is one in $h$, at any time $s\geq t_1$ we can have the local configurations as in Figure~\ref{FigOrderingPaths}.
\begin{figure}
\begin{center}
\psfrag{x}[c]{\small $x$}
\psfrag{h}[c]{$h$}
\psfrag{ht}[c]{$\tilde h$}
\includegraphics[height=3cm]{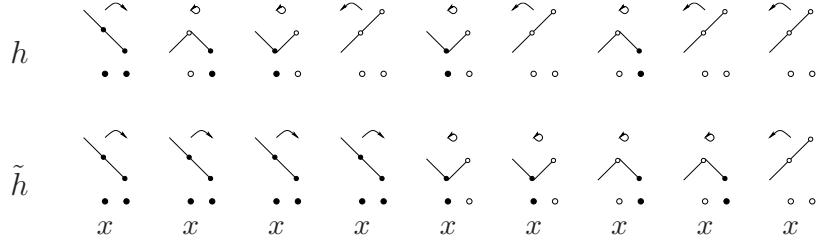}
\caption{Possible local configurations for the height function $h$ and $\tilde h$. The arrows indicate the evolution of the backwards path if they are $x$ when an event of the Poisson point occurs. The black dots represent the TASEP particles associated with the height functions.}
\label{FigOrderingPaths}
\end{center}
\end{figure}
This implies that if $x(t_2)\leq x_2=\tilde x(t_2)$, then the evolution of the backwards path satisfies $x(s)\leq \tilde x(s)$ for all $s\in [t_1,t_2]$. Indeed, if at some point $s$ they become equal, then it can not happen that $x(s_-)>\tilde x(s_-)$, since if $x(s_-)=x(s)+1$, then also $\tilde x(s_-)=\tilde x(s)+1$ and if $\tilde x(s_-)=\tilde x(s)-1$ then also $x(s_-)=x(s)-1$, see Figure~\ref{FigOrderingPaths}.

Furthermore, the law of $\tilde x(s)$ is as the ones of a step initial condition over a time-span $\tau=t_2-t_1=2^{-n+1} t$. Consequently, using Proposition~\ref{PropMidTime} we get
\begin{equation}
\begin{aligned}
\Pb(B_{n,k}\cap A_{n-1}) &\leq \Pb(x_{\rm mid}(2^{-n+1}t)\geq (u_n-u_{n-1})t^{2/3})\\
&\leq C_1 e^{-c_1(u_n-u_{n-1})^2 2^{4(n-1)/3}} = C_1 e^{-c_1 u_1^2 2^{(n-1)/3}}
\end{aligned}
\end{equation}
with our choice of the sequence of $u_n-u_{n-1}$. Using this bound we get that the second term in \eqref{eq3.44} is bounded by $C e^{-c_1 u^2/100}$ for some constant $C$.
By appropriate choice of constants $C_2,c_2$ the result is proven.
\end{proof}

\section{Asymptotic results for the height function}\label{sec:StepAsymp}

\subsection{Asymptotic decoupling}
Here we prove that the height functions are asymptotically independent whenever the positions where we observe the height functions at time $t$ are at distance much larger than $t^{2/3}$.
\begin{thm}\label{thm:DecorrelationHeights}
For $-1 < \alpha <1$ define
\begin{equation}
H_{t}(\alpha)=\frac{h\left( \alpha t, t \right) -\frac12(1+\alpha^2) t}{-2^{-1/3}(1-\alpha^2)^{2/3}t^{1/3}}.
\end{equation}
Let us consider sequences $\{\alpha_{1,t},\ldots,\alpha_{n,t}\}$ satisfying
\begin{equation}
\lim_{t\to\infty} |\alpha_{i,t}-\alpha_{j,t}|t^{1/3}=\infty
\end{equation}
for all $i\neq j$. Then
\begin{equation}
\lim_{t\to\infty}(H_{t}(\alpha_{1,t}),\ldots,H_{t}(\alpha_{n,t}))\stackrel{d}{=}(\xi_1,\ldots,\xi_n),
\end{equation}
where $\xi_1,\ldots,\xi_n$ are independent and (standard) GUE Tracy-Widom distributed.
\end{thm}
\begin{proof}
\emph{Step 1: Use decorrelation of small times.} Choose any sequence $\{\tau_t\}_{t}$ such that $\tau_t/t\to 0$ but $\tau_t/t^{2/3}\to\infty$ as $t\to\infty$.
Define
\begin{equation}
\widetilde H_{t}(\alpha)=\frac{h^{\rm step}_{\alpha \tau_t,\tau_t}(\alpha t,t) -\frac12(1+\alpha^2) (t-\tau_t)}{-2^{-1/3}(1-\alpha^2)^{2/3}t^{1/3}}.
\end{equation}
By Proposition~\ref{PropSlowDec}, for all $\e>0$,
\begin{equation}\label{eq4.5}
\lim_{t\to\infty} \Pb\left(\bigcap_{k=1}^n \{|\widetilde H_{t}(\alpha_{i,t})-H_{t}(\alpha_{i,t})|\leq \e\}\right)=1.
\end{equation}

\emph{Step 2: Use localization of the geodesics.} Now we apply Proposition~\ref{PropLocalization} to the step initial conditions starting from positions $\alpha_{i,\tau_t}\tau_t$ at time $\tau_t$. Let $x_i$ be the backwards geodesic for this process starting from $x_i(t)=\alpha_{i,t}t$. Define the event
\begin{equation}
\Omega^{\rm loc}_{\alpha_{i,t}}=\{\omega: |x(s)-\alpha_{i,t} s|\leq \sigma_t t^{2/3}\textrm{ for all }\tau_t\leq s\leq t\},
\end{equation}
with $\sigma_t$ to be specified later. Let us set $\widetilde H^{\rm loc}_{t}(\alpha)=\widetilde H_{t}(\alpha) \Id_{\Omega^{\rm loc}_{\alpha}}$ for $\alpha\in\{\alpha_{1,t},\ldots,\alpha_{n,t}\}$. Then, Proposition~\ref{PropLocalization}, for any sequence $\sigma_t\to\infty$ as $t\to\infty$,
\begin{equation}\label{eq4.7}
\lim_{t\to\infty}\Pb\left( \widetilde H^{\rm loc}_{t}(\alpha_{i,t})=\widetilde H_{t}(\alpha_{i,t})\right)=1.
\end{equation}
Therefore, if we choose $t^{2/3}\ll \sigma_t t^{2/3} \ll \tau_t \ll t$, then
\begin{equation}
\bigcap_{k=1}^n \{(x,s): |\alpha_{i,t} s-x|\leq \sigma_t t^{2/3}\textrm{ and }\tau_t\leq s\leq t\} = \emptyset
\end{equation}
for all $t\geq t_0$. Thus for all $t$ large enough, the random variables $\widetilde H^{\rm loc}_{t,\alpha_{i,t}}$, $i,\ldots,n$, are independent. Combining this with \eqref{eq4.5} and \eqref{eq4.7} the proof is completed.
\end{proof}

\begin{cor}\label{cor:step1and13}
For $-1 < \alpha <1$ and $\gamma \in \R$, let
\begin{equation}
\xi_{t,\alpha}=\frac{\mathcal{N} \left( \alpha t + \gamma t^{1/3} , t \right) -\frac14(1-\alpha)^2 t + \frac12\gamma (1- \alpha) t^{1/3}}{-2^{-4/3}(1-\alpha^2)^{2/3}t^{1/3}}.
\end{equation}
Then, for distinct $\alpha_1,\ldots,\alpha_n$, it holds
\begin{equation}
\lim_{t\to\infty}(\xi_{t,\alpha_1},\ldots,\xi_{t,\alpha_n})\stackrel{d}{=}(\xi_1,\ldots,\xi_n),
\end{equation}
where $\xi_1,\ldots,\xi_n$ are independent and (standard) GUE Tracy-Widom distributed.
\end{cor}
\begin{proof}
This is obtained by using the relation ${\cal N}(x,t)=\tfrac12(h(x,t)-x)$ and Theorem~\ref{thm:DecorrelationHeights} with $\alpha_{i,t}=\alpha_i+\gamma t^{-2/3}$. The one-point distribution is given by
\begin{equation}\label{eq1}
\lim_{t\to\infty}\frac{\mathcal{N} \left( \alpha_{i,t} t, t \right) -\frac14(1-\alpha_{i,t})^2 t}{-2^{-4/3}(1-\alpha_{i,t}^2)^{2/3}t^{1/3}} \stackrel{d}{=} \xi_i.
\end{equation}
With the choice of $\alpha_{i,t}$, the expansion of the law of large number up to $O(t^{1/3})$ is
\begin{equation}
\tfrac14(1-\alpha_{i,t})^2 t=\tfrac14(1-\alpha_i)^2t - \tfrac12 \gamma ( 1-\alpha_i)t^{1/3}+O(t^{-1/3}).
\end{equation}
\end{proof}

\begin{cor}\label{cor:stepDeltaLarge}
For $-1 < \alpha <1$, $\beta, \gamma \in \R$, and $2/3 < \delta < 1$. Define
\begin{equation}
\zeta_{t,\beta}=\frac{\mathcal{N} \left( \alpha t + \beta t^{\delta} + \gamma t^{4/3 - \delta} , t \right)-\frac{(1-\alpha)^2}{4} t + \frac{(1-\alpha)\beta}{2} t^{\delta}+\frac{(1-\alpha)\gamma t^{4/3-\delta}}{2}-\frac{\beta^2 t^{2\delta-1}}{4}-\frac{\beta\gamma t^{1/3}}{2}}{- \frac{(1-\alpha^2)^{2/3}}{2^{4/3}} t^{1/3}}.
\end{equation}
Then, for distinct $\beta_1,\ldots,\beta_n$,
\begin{equation}
\lim_{t\to\infty}(\zeta_{t,\beta_1},\ldots,\zeta_{t,\beta_n})\stackrel{d}{=}(\zeta_{1},\ldots,\zeta_{n}),
\end{equation}
where $\zeta_{1},\ldots,\zeta_{n}$ are independent and (standard) GUE Tracy-Widom distributed.
\end{cor}
\begin{proof}
This is obtained by using the relation ${\cal N}(x,t)=\tfrac12(h(x,t)-x)$ and Theorem~\ref{thm:DecorrelationHeights} with $\alpha_{i,t}=\alpha+ \beta_i t^{\delta-1} + \gamma t^{1/3 - \delta}$, for which the expansion of the law of large number up to $O(t^{1/3})$ is
\begin{equation}
\tfrac14(1-\alpha_{i,t})^2 t=\tfrac14(1-\alpha)^2t-\tfrac12\beta_i(1-\alpha)t^\delta-\tfrac12\gamma(1-\alpha)t^{4/3-\delta}+\tfrac12\beta_i\gamma t^{1/3}+\tfrac14\beta_i^2 t^{2\delta-1}+o(t^{1/3}).
\end{equation}
\end{proof}

\subsection{Local stationarity}
Now we prove local Gaussian increments for the height function.
\begin{thm}\label{thm:LocalGaussianHeight}
For any $-1< \alpha <1$, $\beta, \gamma_1,\gamma_2 \in \R$, with $\gamma_1< \gamma_2$, and $0< \delta< 2/3$, define
$x_i=\alpha t+\beta t^{1-\delta/2}+\gamma_i t^\delta$. Then we have
\begin{equation}
\lim_{t\to\infty}\frac{h^{\rm step}(x_2,t)-h^{\rm step}(x_1,t)-(\gamma_2-\gamma_1)(\alpha t^\delta+\beta t^{\delta/2})}{\sqrt{(1-\alpha^2)(\gamma_2-\gamma_1)}t^{\delta/2}} \stackrel{d}{=} G(0,1),
\end{equation}
where $G(0,1)$ is centered Gaussian random variable with variance $1$.
\end{thm}
\begin{proof}
Let us set $\rho=\tfrac12(1-\alpha-\beta t^{-\delta/2})$, i.e., satisfying $(1-2\rho)t=\alpha t+\beta t^{1-\delta/2}$. We define for $\kappa>0$ the densities
\begin{equation}
\rho_\pm=\rho \pm \kappa t^{-1/3}.
\end{equation}
Define the sets
\begin{equation}
G_+ = \{x_{\rho_+}(0)<0 | x_{\rho_+}(t)=x_i\},\quad G_- = \{x_{\rho_-}(0)>0 | x_{\rho_-}(t)=x_i\},
\end{equation}
where $x_i$ are the starting points for the backwards geodesic in the stationary case with density $\rho_\pm$.

By Proposition~\ref{PropLocalStatInitialPoint}, for the stationary initial condition with density $\rho_+$ and using translation invariance we have
\begin{equation}
\begin{aligned}
1-C e^{-c M^2}& \leq \Pb(|x_{\rho_+}(0)|< M t^{2/3} | x_{\rho_+}(t)=(1-2\rho_+)t) \\
&\leq \Pb(x_{\rho_+}(0)<M t^{2/3} | x_{\rho_+}(t)=(1-2\rho_+)t)\\
&= \Pb(x_{\rho_+}(0)+x_i-(1-2\rho_+)t<M t^{2/3} | x_{\rho_+}(t)=x_i).
\end{aligned}
\end{equation}
But $x_i-(1-2\rho_+)t=\gamma_i t^{\delta}+2\kappa t^{2/3}$, thus by choosing $M=\kappa$ and $t$ large enough, $x_i-(1-2\rho_+)< M t^{2/3}$. This gives
\begin{equation}
\Pb(G_+)=\Pb(x_{\rho_+}(0)<0 | x_{\rho_+}(t)=x_i)\geq 1-C e^{-c \kappa^2}.
\end{equation}
Similarly, we have
\begin{equation}
\Pb(G_-)=\Pb(x_{\rho_-}(0)>0 | x_{\rho_-}(t)=x_i)\geq 1-C e^{-c \kappa^2}.
\end{equation}
On the set $G=G_+\cap G_-$ we apply the comparison lemma, Lemma~\ref{lemComparison}, since we know that for step initial condition the backwards geodesics can be taken such that it reaches the origin, see Remark~\ref{remGeodesics}. This gives
\begin{equation}
h_{\rho_+}(x_2,t)-h_{\rho_+}(x_1,t)\geq h^{\rm step}(x_2,t)-h^{\rm step}(x_1,t) \geq h_{\rho_-}(x_2,t)-h_{\rho_-}(x_1,t),
\end{equation}
where $h_{\rho}$ denotes the height function for the stationary initial condition with density $\rho$, and $h^{\rm step}$ the height function with step initial condition.

In the stationary cases, the height difference is a sum of independent random variables:
\begin{equation}
h_\rho(x_2,t)-h_\rho(x_1,t)=\sum_{k=x_1}^{x_2} X_k,
\end{equation}
with $\Pb(X_k=1)=1-\rho$ and $\Pb(X_k=-1)=\rho$. Thus, by the central limit theorem, we have
\begin{equation}
\frac{h_{\rho_+}(x_2,t)-h_{\rho_+}(x_1,t)-(1-2\rho)(\gamma_2-\gamma_1)t^{\delta}+2\kappa(\gamma_2-\gamma_1)t^{\delta-1/3}}{\sqrt{4\rho(1-\rho)(\gamma_2-\gamma_1)}t^{\delta/2}} \stackrel{t\to\infty}{\Longrightarrow} G(0,1).
\end{equation}
Replace $\kappa$ by a sequence $\{\kappa_t\}_t$ such that $\kappa_t\to\infty$ but $\kappa t^{\delta/2-1/3}\to 0$ as $t\to\infty$. This is possible since $\delta<2/3$. Then the $\kappa$-dependent drift becomes irrelevant. Thus
\begin{equation}
\frac{h_{\rho_+}(x_2,t)-h_{\rho_+}(x_1,t)-(1-2\rho)(\gamma_2-\gamma_1)t^{\delta}}{\sqrt{4\rho(1-\rho)(\gamma_2-\gamma_1)}t^{\delta/2}} \stackrel{t\to\infty}{\Longrightarrow} G(0,1).
\end{equation}
In the same way we obtain the convergence of the increment of the middle distribution to the distribution. The choice of $\kappa_t\to\infty$ implies also that $\Pb(G)\to 1$. Therefore also the sandwiched height function have asymptotically the same distribution function (apply for example Lemma~\ref{LemCorwin}). Replacing the value of $\rho$ in terms of the (limits of) $\alpha,\beta$ the claimed result is proven.
\end{proof}

\begin{remark}
In our work on second class particle we are going to use only the convergence to a Gaussian law. However, the estimates also gives tightness of the process, which locally converges weakly to Brownian motion. This sandwiching procedure was first used in~\cite{CP15b} to prove tightness of the process, see~\cite{Pim17,FO17} for further applications.
\end{remark}

The same statement in terms of the random variable $\cal N$ is the following.
\begin{cor}\label{cor:stepDeltaSmall}
For any $-1< \alpha <1$, $\beta, \gamma_1,\gamma_2 \in \R$, with $\gamma_1< \gamma_2$, and $0< \delta< 2/3$, define
$x_i=\alpha t+\beta t^{1-\delta/2}+\gamma_i t^\delta$. Then we have
\begin{equation}
\lim_{t\to\infty}\frac{\mathcal{N} (x_2,t)- \mathcal{N}(x_1,t)-\frac12(\gamma_2-\gamma_1)((\alpha-1)t^\delta+\beta t^{\delta/2})}{\frac12\sqrt{(1-\alpha^2)(\gamma_2-\gamma_1)}t^{\delta/2}}\stackrel{d}{=}G(0,1),
\end{equation}
where $G(0,1)$ is a centered Gaussian random variable with variance $1$.
\end{cor}
\begin{proof}
It follows from Theorem~\ref{thm:LocalGaussianHeight} and the relation ${\cal N}(x,t)=\tfrac12(h(x,t)-x)$.
\end{proof}

\appendix

\section{Some known results on the step initial condition}
For $\alpha\in(-1,1)$ fixed, define the rescaled height function for step initial condition by
\begin{equation}
u\mapsto h^{\rm resc}_{t}(u)=\frac{h(\alpha t+u t^{2/3},t)-(\tfrac12(1+\alpha^2)t+\alpha \kappa_h ut^{2/3}-\tfrac12 \kappa_h^2 u^2 t^{1/3})}{-\kappa_v t^{1/3}},
\end{equation}
with
\begin{equation}
\kappa_v=2^{-1/3}(1-\alpha^2)^{2/3},\quad \kappa_h= 2^{1/3}(1-\alpha^2)^{1/3}.
\end{equation}

\begin{thm}\label{Thm:CvgAiry}
Let $h(x,t)$ be TASEP height function starting with step initial condition. Then, for any given $M>0$, we have weak convergence on ${\cal C}([-M,M])$:
\begin{equation}
\lim_{t\to\infty}h^{\rm resc}_{t}(u) = {\cal A}_2(u).
\end{equation}
In particular,
\begin{equation}
\lim_{t\to\infty} \Pb(h^{\rm resc}_t(u)\leq s)=F_{\rm GUE}(s).
\end{equation}
\end{thm}
The one-point convergence is proven in Theorem~1.6 of~\cite{Jo00b}. The convergence to the Airy$_2$ process in terms of finite-dimensional distributions are special cases of LPP models~\cite{BP07,SI07} and in terms TASEP of particle positions in~\cite{BF07}. In the LPP framework tightness is proven in~\cite{FO17}, which implies weak convergence to the Airy$_2$ process (see Corollary~2.4 of~\cite{FO17}). Furthermore, a functional slow-decorrelation result is proven Theorem~4.1 of~\cite{FO17}, which implies the weak convergence also for the height function point of view (see~\cite{CFP10b,BFP09} to see how slow-decorrelation allows to deduce results on the height functions from results proven in the LPP picture or in the TASEP particle position point of view).

Theorem~\ref{Thm:CvgAiry} rewritten for the observable $\cal N$ is the following.
\begin{prop}\label{prop:TasepAiry}
For any given $M>0$, the rescaled $\cal N$ converges weakly to the Airy$_2$ process on the space of continuous functions ${\cal C}([-M,M])$:
\begin{equation}
\lim_{t \to \infty} \frac{ \mathcal{N} \left( 2 u (t/2)^{2/3}, t \right) - t/4 + u (t/2)^{2/3} - u^2 t^{1/3} 2^{-4/3} }{ - t^{1/3} 2^{-4/3}} = \mathcal{A}_2 (u).
\end{equation}
\end{prop}

Upper and lower bounds on the one-point distributions are also known.
\begin{prop}
Bound on upper tail: for given $s_0>0$ and $t_0>0$, there exists constants $C,c$ such that
\begin{equation}\label{eqBoundUpper}
\Pb(h^{\rm resc}_{t}(u)\geq s)\leq C e^{-c s}
\end{equation}
for all $t\geq t_0$ and $s\geq s_0$.\\
Bound on lower tail: for given $s_0>0$ and $t_0>0$, there exist constants $C,c$ such that
\begin{equation}\label{eqBoundLower}
\Pb(h^{\rm resc}_{t}(u)\leq s)\leq C e^{-c |s|^{3/2}}
\end{equation}
for all $t\geq t_0$ and $s\leq -s_0$.
\end{prop}
The constants $C,c$ can be chosen uniformly for $\alpha$ in a bounded set of $(-1,1)$.  Using the relation with the Laguerre ensemble of random matrices (Proposition~6.1 of~\cite{BBP06}), or to TASEP, one sets the distribution is given by a Fredholm determinant. An exponential decay of its kernel leads directly to the upper tail. See e.g.\ Lemma~1 of~\cite{BFP12} for an explicit statement. The lower tail was proven in~\cite{BFP12} (Proposition~3 together with (56)).

Here is a probabilistic lemma used in the slow-decorrelation type theorems.
\begin{lem}[Lemma~4.1 of~\cite{BC09}]\label{LemCorwin}
Consider two sequences of random variables $\{X_n\}$ and $\{\tilde X_n\}$ such that for each $n$, $X_n$ and $\tilde X_n$ are defined on the same probability space. If $X_n\geq \tilde X_n$ and $X_n\Rightarrow D$ as well as $\tilde X_n\Rightarrow D$, then $X_n-\tilde X_n$ converges to zero in probability. Conversely, if $\tilde X_n\Rightarrow D$ and $X_n-\tilde X_n$ converges to $0$ in probability, then $X_n\Rightarrow D$ as well.
\end{lem}


\begin{thebibliography}{10}

\bibitem{AAV11}
G.~Amir, O.~Angel, and B.~Valk\'{o}.
\newblock {The TASEP speed process}.
\newblock {\em Ann. Probab.}, 39:1205--1242, 2011.

\bibitem{AHR09}
O.~Angel, A.~Holroyd, and D.Romik.
\newblock The oriented swap process.
\newblock {\em Ann. Probab.}, 37:1970--1998, 2009.

\bibitem{BBP06}
J.~Baik, G.~{Ben Arous}, and S.~P\'ech\'e.
\newblock Phase transition of the largest eigenvalue for non-null complex
  sample covariance matrices.
\newblock {\em Ann. Probab.}, 33:1643--1697, 2006.

\bibitem{BFP09}
J.~Baik, P.L. Ferrari, and S.~P{\'e}ch{\'e}.
\newblock {Limit process of stationary TASEP near the characteristic line}.
\newblock {\em Comm. Pure Appl. Math.}, 63:1017--1070, 2010.

\bibitem{BFP12}
J.~Baik, P.L. Ferrari, and S.~P{\'e}ch{\'e}.
\newblock {Convergence of the two-point function of the stationary TASEP}.
\newblock In {\em {Singular Phenomena and Scaling in Mathematical Models}},
  pages 91--110. Springer, 2014.

\bibitem{BSS14}
R.~Basu, V.~Sidoravicius, and A.~Sly.
\newblock {Last passage percolation with a defect line and the solution of the
  Slow Bond Problem}.
\newblock {\em preprint: arXiv:1408.346}, 2014.

\bibitem{BC09}
G.~{Ben Arous} and I.~Corwin.
\newblock {Current fluctuations for TASEP: a proof of the Pr\"ahofer-Spohn
  conjecture}.
\newblock {\em Ann. Probab.}, 39:104--138, 2011.

\bibitem{BB19}
A.~Borodin and A.~Bufetov.
\newblock Color-position symmetry in interacting particle systems.
\newblock {\em preprint: arXiv:1905.04692}, 2019.

\bibitem{BF07}
A.~Borodin and P.L. Ferrari.
\newblock {Large time asymptotics of growth models on space-like paths I:
  PushASEP}.
\newblock {\em Electron. J. Probab.}, 13:1380--1418, 2008.

\bibitem{BP07}
A.~Borodin and S.~P\'ech\'e.
\newblock {Airy Kernel with Two Sets of Parameters in Directed Percolation and
  Random Matrix Theory}.
\newblock {\em J. Stat. Phys.}, 132:275--290, 2008.

\bibitem{B20}
A.~Bufetov.
\newblock {Interacting particle systems and random walks on Hecke algebras}.
\newblock {\em preprint, arXiv:2003.02730}, 2020.

\bibitem{CP15b}
E.~Cator and L.~Pimentel.
\newblock On the local fluctuations of last-passage percolation models.
\newblock {\em Stoch. Proc. Appl.}, 125:879--903, 2015.

\bibitem{CFP10b}
I.~Corwin, P.L. Ferrari, and S.~P{\'e}ch{\'e}.
\newblock {Universality of slow decorrelation in KPZ models}.
\newblock {\em Ann. Inst. H. Poincar\'e Probab. Statist.}, 48:134--150, 2012.

\bibitem{FK95}
P.A. Ferrari and C.~Kipnis.
\newblock Second class particles in the rarefaction fan.
\newblock {\em Ann. Inst. H. Poincar\'e}, 31:143--154, 1995.

\bibitem{FMP09}
P.A. Ferrari, J.B. Martin, and L.P.R. Pimentel.
\newblock A phase transition for competition interfaces.
\newblock {\em Ann. Appl. Probab.}, 19:281--317, 2009.

\bibitem{FP05B}
P.A. Ferrari and L.P.R. Pimentel.
\newblock Competition interfaces and second class particles.
\newblock {\em Ann. Probab.}, 33:1235--1254, 2005.

\bibitem{Fer08}
P.L. Ferrari.
\newblock {Slow decorrelations in KPZ growth}.
\newblock {\em J. Stat. Mech.}, page P07022, 2008.

\bibitem{Fer18}
P.L. Ferrari.
\newblock {Finite GUE distribution with cut-off at a shock}.
\newblock {\em J. Stat. Phys.}, 172:505--521, 2018.

\bibitem{FGN17}
P.L. Ferrari, P.~Ghosal, and P.~Nejjar.
\newblock {Limit law of a second class particle in TASEP with non-random
  initial condition}.
\newblock {\em Ann. Inst. Henri Poincar\'e Probab. Statist.}, 55:1203--1225,
  2019.

\bibitem{FN13}
P.L. Ferrari and P.~Nejjar.
\newblock {Anomalous shock fluctuations in TASEP and last passage percolation
  models}.
\newblock {\em Probab. Theory Relat. Fields}, 161:61--109, 2015.

\bibitem{FN16}
P.L. Ferrari and P.~Nejjar.
\newblock {Fluctuations of the competition interface in presence of shocks}.
\newblock {\em ALEA, Lat. Am. J. Probab. Math. Stat.}, 14:299--325, 2017.

\bibitem{FN19}
P.L. Ferrari and P.~Nejjar.
\newblock {Statistics of TASEP with three merging characteristics}.
\newblock {\em J. Stat. Phys. (online first)}, 2019.

\bibitem{FO17}
P.L. Ferrari and A.~Occelli.
\newblock {Universality of the GOE Tracy-Widom distribution for TASEP with
  arbitrary particle density}.
\newblock {\em Eletron. J. Probab.}, 23(51):1--24, 2018.

\bibitem{G20}
P.~Galashin.
\newblock Symmetries of stochastic colored vertex models.
\newblock {\em preprint, arXiv:2003.06330}, 2020.

\bibitem{Har78}
T.E. Harris.
\newblock Additive set-valued markov processes and pharical methods.
\newblock {\em Ann. Probab.}, 6:355--378, 1878.

\bibitem{Har72}
T.E. Harris.
\newblock {Nearest-neighbor Markov interaction processes on multidimensional
  lattices}.
\newblock {\em Adv. Math.}, 9:66--89, 1972.

\bibitem{Hol70}
R.~Holley.
\newblock A class of interactions in an infinite particle system.
\newblock {\em Adv. Math.}, 5:291--309, 1970.

\bibitem{SI07}
T.~Imamura and T.~Sasamoto.
\newblock Dynamical properties of a tagged particle in the totally asymmetric
  simple exclusion process with the step-type initial condition.
\newblock {\em J. Stat. Phys.}, 128:799--846, 2007.

\bibitem{Jo00b}
K.~Johansson.
\newblock Shape fluctuations and random matrices.
\newblock {\em Comm. Math. Phys.}, 209:437--476, 2000.

\bibitem{Lig72}
T.M. Liggett.
\newblock Existence theorems for infinite particle systems.
\newblock {\em Trans. Amer. Math. Soc.}, 165:471--481, 1972.

\bibitem{Li99}
T.M. Liggett.
\newblock {\em Stochastic interacting systems: contact, voter and exclusion
  processes}.
\newblock Springer Verlag, Berlin, 1999.

\bibitem{N17}
P.~Nejjar.
\newblock {Transition to shocks in TASEP and decoupling of last passage times}.
\newblock {\em ALEA, Lat. Am. J. Probab. Math. Stat.}, 15:1311--1334, 2018.

\bibitem{N19}
P.~Nejjar.
\newblock {${\rm GUE}\times{\rm GUE}$ limit law at hard shocks in ASEP}.
\newblock {\em preprint: arXiv:1906.07711}, 2019.

\bibitem{Pim17}
L.P.R.~Pimentel.
\newblock {Local Behavior of Airy Processes}.
\newblock {\em J. Stat. Phys.}, 173:1614--1638, 2018.

\bibitem{PS01}
M.~Pr{\"a}hofer and H.~Spohn.
\newblock Current fluctuations for the totally asymmetric simple exclusion
  process.
\newblock In V.~Sidoravicius, editor, {\em In and out of equilibrium}, Progress
  in Probability. Birkh{\"a}user, 2002.

\bibitem{Sep98c}
T.~Sepp{\"a}l{\"a}inen.
\newblock Coupling the totally asymmetric simple exclusion process with a
  moving interface.
\newblock {\em Markov Proc. Rel. Fields}, 4 no.4:593--628, 1998.

\end{thebibliography}

\end{document}